\documentclass[11pt,a4paper]{amsart}
\usepackage{graphicx}
\usepackage{amsmath,amsfonts}
\usepackage{amsthm,amssymb,latexsym}
\usepackage{hyperref}
\usepackage[active]{srcltx}
\usepackage[usenames]{color}
\usepackage[utf8]{inputenc} 
\DeclareMathOperator{\rank}{rank}

\newcommand{\bfs}{\boldsymbol}
\newtheorem{theorem}{Theorem}[section]

\newtheorem{lemma}[theorem]{Lemma}

\newtheorem*{fact*}{Fact}
\newtheorem{proposition}[theorem]{Proposition}
\theoremstyle{definition}

\newtheorem{remark}[theorem]{Remark}
\numberwithin{equation}{section}

\newcommand{\N}{\mathbb N}

\newcommand{\A}{\mathbb A}
\newcommand{\F}{\mathbb F}
\newcommand{\K}{\mathbb K}

\newcommand{\Pp}{\mathbb P}

\newcommand{\fq}{\F_{\hskip-0.7mm q}}

\newcommand{\cfq}{\overline{\F}_{\hskip-0.7mm q}}
\def\ifm#1#2{\relax \ifmmode#1\else#2\fi}

\newcommand{\klk}    {\ifm {,\ldots,} {$,\ldots,$}}

\begin{document}


\title[System of diagonal equations and the subset sum problem]{An approach to the moments subset sum problem through systems of diagonal equations over finite fields}

\author[Juan Francisco Gottig]{Juan Francisco Gottig${}^{1,2}$}

\author[Mariana Pérez]{Mariana Pérez${}^{1,2}$}

\author[Melina Privitelli]{Melina Privitelli${}^{1,2}$}

\address{${}^{1}$Universidad Nacional de Hurlingham, Instituto de Tecnolog\'ia e
Ingenier\'ia\\ Av. Gdor. Vergara 2222 (B1688GEZ), Villa Tesei,
Buenos Aires, Argentina}
\email{\{francisco.gottig,mariana.perez,melina.privitelli\}@unahur.edu.ar}
\address{${}^{2}$ Consejo Nacional de Investigaciones Científicas y Técnicas (CONICET),
Ar\-gentina}

\thanks{The authors were partially supported by the grants
PIUNAHUR 2023 80020230100
003HU, PIBAA 2022-2023-28720210100460CO, PIBAA 2022-2023-28720210100422CO and UNGS
30/1146.}


\begin{abstract}
Let $\fq$ be the finite field of $q$ elements, for a given subset $D\subset \fq$, $m\in \N$, an integer $k\leq |D|$ and $\bfs{b}\in \fq^m$ we are interested in determining the existence of a subset $S\subset D$ of cardinality $k$ such that  $\sum_{a\in S}a^i=b_i$ for $i=1,\ldots, m$. This problem is known as the {\it{moment subset sum problem}} and it is $NP$-complete for a general $D$.

We make a novel approach of this problem trough algebraic geometry tools analyzing the underlying variety and employing combinatorial techniques to estimate the number of $\mathbb{F}_q$-rational points on certain varieties.

We managed to give estimates on the number of $\mathbb{F}_q$-rational points on certain diagonal equations and use this results to give estimations and existence results for the \textit{subset sum problem}.
\end{abstract}

\maketitle

\section{Introduction}

Let $\fq$ be the finite field of $q=p^s$ elements, let $D$ be a nonempty subset of $\fq$ of cardinality $|D|$ and integer $k$ with $1\leq k \leq |D|$. Let $m\in \N$ and $\bfs{b}:=(b_1,\ldots,b_m)\in \fq^m$. We denote by $N_m(k,\bfs{b},D)$ as the number of subsets $S\subset D$ with cardinal $k$ such that for $i=1,\ldots, m$, $\sum_{a\in S}a^i=b_i$, namely

$$N_m(k,\bfs{b},D):=\#\{S\subset D:\,\, \sum_{a\in S}a^i=b_i\, ,i=1,\ldots,m,\,|S|=k\}.$$

The problem is to determine whether $N_m(k,\bfs{b}, D)$ is  positive. This problem is known as the {\textit{m}}--th moment {\textit{k}}--subset sum problem ({\textit{m}}--th moment {\textit{k}}--SSP) and it has relevant applications in coding theory and cryptography. For example, when $D=\fq$ it is related to the deep hole problem of extended Reed-Solomon codes, (see, e.g., \cite{LiWa08}, \cite{Nu2019}). On the other hand, in \cite{Me1978} the authors introduced a public key cryptosystem that relies on a modified version of this problem when $m=1$. There are many results in the case $m=1$. In \cite{LiWa08} the authors provided explicit formulas for $N_1(k,\bfs{b},D)$ if $D=\fq$ or $D=\fq \setminus \{0\}$ using a combinatorial method. Furthermore, in the same article there are asymptotic formulas if $\fq \setminus D$ is small. When $D:=\{x^2: x\in \fq\}$ \cite{Wang2017} provide an explicit formula of $N_1(k,\bfs{b},D)$. In \cite{Keti2016} the authors studied the case when $D$ is the image of a Dickson polynomial of degree $d$ with parameter $a\in \fq$. The main difficulty of this problem comes from the subset $D$, specially when it lacks algebraic structure. For a general $D$ and $m=1$, it has been shown that the problem is NP-complete, see \cite{Cormen09}. On the other hand, there are not much results for $m>1$. An explicit combinatorial formula for
$N_m(k, \bfs{b},D)$ is obtained in \cite{Nu2019} when $m = 2$ and $D = \fq$. In \cite{LiWa08} the authors provided an asymptotic formula for a general $m$ and $D=\fq$. Indeed, they obtained the following estimate on the number $N_m(k,\bfs{b}):=N_m(k,\bfs{b},\fq)$.
\begin{theorem}\cite[Theorem 1.3]{Wan2010}
For all $\bfs{b} \in \fq^m$,
$$\Big|N_m(k,\bfs{b})- \frac{1}{q^m} \binom{q}{k}\Big| \leq \binom{q/p+(m-1)\sqrt{q}+k-1}{k}.$$
\end{theorem}
Observe that, in the theorem above, the error term is of order $\mathcal{O}(p^{
(s-1)k
})$, while in this article, we get an estimate on $N_m(k,\bfs{b})$ with error term of order $\mathcal{O}(p^{
sk/2
})$. Furthermore, when $p$ divides $k$ we obtain one more term in the asymptotic development of $N_m(k,\bfs{b})$ in terms of $q$, see Theorems \ref{teo: estimación Nm(k,b) p no divide a k} and \ref{teo: estimación Nm(k,b) p  divide a k}. 
\begin{theorem}
Let $m, k$ be positive integers such that  $k \leq 2q^{0.9}-\sqrt{q}+1$ and $m \leq \frac{k}{20}$. Suppose that $q>2^{21}$ and $p\geq 5$. We have
\begin{enumerate}
    \item if $p$ does not divide $k$, then
    \[
\left |N_m(k,\boldsymbol{b})-\frac{1}{q^m}{q\choose k}\right|\leq M \cdot (-1)^k \binom{-\sqrt{q}}{k},
\]
\item if $p$ divide $k$, then
\[
\left |N_m(k,\boldsymbol{b})-\frac{1}{q^m}\Big({q\choose k} +(-1)^{k+k/p} v(\bfs{b}) \binom{q/p}{k/p}\Big)\right|\leq M \cdot (-1)^k \binom{-\sqrt{q}}{k},
\]
\end{enumerate}
where $v(b)=q^m-1$  if $\bfs{b}=0$, $v(\bfs{b})=-1$ if $\bfs{b} \neq 0$ and $M=3^4\cdot 2^{m-1}(3+d_mm)^{k+1}.$
\end{theorem}
Furthermore, in \cite{Wan2010} the authors provided an existence's result.
\begin{theorem}\cite[Theorem 1.4]{Wan2010}
For any $\varepsilon>0$, there is a constant $c_{\varepsilon}>0$ such that if $m<\varepsilon \sqrt{k}$ and $4\varepsilon^2\ln^2{q}<k<c_{\varepsilon
}q$, then $N_m(k,\bfs{b})>0$ for all $\bfs{b}\in \fq^m$.
\end{theorem}
Observe that $q$ can be linear in k in the above theorem and the relation between $m$ and $k$ is $m=\mathcal{O}(k^{1/2})$. In this paper,  we get a linear relation between $m$ and $k$ (see  Theorem \ref{teo: existencia con Brun}).

\begin{theorem} 
    Let $D=\fq$, $ m \leq \frac{k-25}{50}$ and $ k\leq q^{0,24}$
    then $N_m(k,\bfs{b})>0$ for all $\bfs{b}\in \fq^m$.
\end{theorem}
In \cite{Marino2020} the autors exhibit a deterministic polynomial time algorithm for the {\textit{m}}--th moment {\textit{k}}--SSP  for a fixed $m$ when the set $D$ is the image set of a monomial or a Dickson polynomial of any degree $n$ and for small  $k$ range such as $k<3m+1$.  An open problem is to study the \textit{m}--th moment {\textit{k}}--SSP  for a general set $D$. In this paper, we provide conditions over $q$, $m$ and $k$ when $D$ is the image set of a polynomial $f=a_nT^n+\cdots + a_2T^2$.
Our approach is different to the one used in \cite{LiWa08}, \cite{Wan2010}, \cite{Nu2019} and \cite{Marino2020}.  We related the \textit{m}--th moment {\textit{k}}--SSP with the problem of estimating the number of $\fq$--rational solutions of a system of diagonal equations.

There are few results for this type of systems in the literature.
In \cite{Cao2016} and \cite{CRGF08}, the authors present exact counting results for specific cases of systems involving two polynomials. In \cite{Spackman79} and \cite{Spackman81}, there are estimations with matching exponents in corresponding columns. And in \cite{PP22} there are estimations with matching exponents in corresponding columns and $\boldsymbol{b}\in \mathbb{F}_q[X_1,\ldots,X_k]^m$. In this sense, we focus on a new type of system with exponents matching across rows.

The paper is organized as follows.
 In Section 2 we briefly recall the notions and notations of algebraic geometry we use. Section 3 is devoted to estimate the number of $\fq$--rational solutions of systems of diagonal equations, this is solutions with coordinates in $\fq$.
In Section 4 we obtain  estimates on the number $N_m(k,\bfs{b})$ and we provide conditions under which this number is positive.
In Section 5 we study the \textit{m}--th moment {\textit{k}}--SSP problem when $D$ is the image set of a polynomial $f=a_nT^n+\cdots + a_2T^2$. Finally, in Section 6 we recall some results related to combinatorial numbers.

\section{Basic notions of algebraic geometry}
%
In this section we collect the basic definitions and facts of
algebraic geometry that we need for our results. We use standard
notions and notations which can be found in, e.g., \cite{Kunz85},
\cite{Shafarevich94}.

Let $\K$ be any of the fields $\fq$ or $\cfq$, the closure of $\fq$. We denote by $\A^r$
the affine $r$--dimensional space $\cfq{\!}^{r}$ and by $\Pp^r$ the
projective $r$--dimensional space over $\cfq$. Both spaces
are endowed with their respective Zariski topologies over $\K$, for
which a closed set is the zero locus of a set of polynomials of
$\K[X_1,\ldots, X_r]$, or of a set of homogeneous polynomials of
$\K[X_0,\ldots, X_r]$.

A subset $V\subset \Pp^r$ is a {\em projective variety defined over}
$\K$ (or a projective $\K$--variety for short) if it is the set of
common zeros in $\Pp^r$ of homogeneous polynomials $F_1,\ldots, F_m
\in\K[X_0 ,\ldots, X_r]$. Correspondingly, an {\em affine variety of
	$\A^r$ defined over} $\K$ (or an affine $\K$--variety) is the set of
common zeros in $\A^r$ of polynomials $F_1,\ldots, F_{m} \in
\K[X_1,\ldots, X_r]$. We think a projective or affine $\K$--variety
to be endowed with the induced Zariski topology. We shall denote by
$\{F_1=0,\ldots, F_m=0\}$ or $V(F_1,\ldots,F_m)$ the affine or
projective $\K$--variety consisting of the common zeros of
$F_1,\ldots, F_m$.

In the remaining part of this section, unless otherwise stated, all
results referring to varieties in general should be understood as
valid for both projective and affine varieties.

A $\K$--variety $V$ is {\em irreducible} if it cannot be expressed
as a finite union of proper $\K$--subvarieties of $V$. Further, $V$
is {\em absolutely irreducible} if it is $\cfq$--irreducible as a
$\cfq$--variety. Any $\K$--variety $V$ can be expressed as an
irredundant union $V=\mathcal{C}_1\cup \cdots\cup\mathcal{C}_s$ of
irreducible (absolutely irreducible) $\K$--varieties, unique up to
reordering, called the {\em irreducible} ({\em absolutely
	irreducible}) $\K$--{\em components} of $V$.

For a $\K$--variety $V$ contained in $\Pp^r$ or $\A^r$, its {\em
	defining ideal} $I(V)$ is the set of polynomials of $\K[X_0,\ldots,
X_r]$, or of $\K[X_1,\ldots, X_r]$, vanishing on $V$. The {\em
	coordinate ring} $\K[V]$ of $V$ is the quotient ring
$\K[X_0,\ldots,X_r]/I(V)$ or $\K[X_1,\ldots,X_r]/I(V)$. The {\em
	dimension} $\dim V$ of $V$ is the length $n$ of a longest chain
$V_0\varsubsetneq V_1 \varsubsetneq\cdots \varsubsetneq V_n$ of
nonempty irreducible $\K$--varieties contained in $V$. 
We say that $V$ has {\em pure dimension} $n$ if every irreducible
$\K$--component of $V$ has dimension $n$. If $W$ is a subvariety of $V$, then the number $\dim V - \dim W$ is
called the {\it{codimension}} of $W$ in $V$.
A $\K$--variety of $\Pp^r$
or $\A^r$ of pure dimension $r-1$ is called a $\K$--{\em
	hypersurface}. A $\K$--hypersurface of $\Pp^r$ (or $\A^r$) can also
be described as the set of zeros of a single nonzero polynomial of
$\K[X_0,\ldots, X_r]$ (or of $\K[X_1,\ldots, X_r]$).

The {\em degree} $\deg V$ of an irreducible $\K$--variety $V$ is the
maximum of $|V\cap L|$, considering all the linear spaces $L$ of
codimension $\dim V$ such that $|V\cap L|<\infty$. More generally,
following \cite{Heintz83} (see also \cite{Fulton84}), if
$V=\mathcal{C}_1\cup\cdots\cup \mathcal{C}_s$ is the decomposition
of $V$ into irreducible $\K$--components, we define the degree of
$V$ as
$$\deg V:=\sum_{i=1}^s\deg \mathcal{C}_i.$$
The degree of a $\K$--hypersurface $V$ is the degree of a polynomial
of minimal degree defining $V$. 
We shall use the following {\em B\'ezout inequality} (see
\cite{Fulton84, Heintz83, Vogel84}): if $V$ and $W$
are $\K$--varieties of the same ambient space, then
\begin{equation}\label{eq: Bezout}
\deg (V\cap W)\le \deg V \cdot \deg W.
\end{equation}

Let $V\subset\A^r$ be a $\K$--variety, $I(V)\subset \K[X_1,\ldots,
X_r]$ its defining ideal and $x$ a point of $V$. The {\em dimension}
$\dim_xV$ {\em of} $V$ {\em at} $x$ is the maximum of the dimensions
of the irreducible $\K$--components of $V$ containing $x$. If
$I(V)=(F_1,\ldots, F_m)$, the {\em tangent space} $\mathcal{T}_xV$
to $V$ at $x$ is the kernel of the Jacobian matrix $(\partial
F_i/\partial X_j)_{1\le i\le m,1\le j\le r}(x)$ of $F_1,\ldots, F_m$
with respect to $X_1,\ldots, X_r$ at $x$. We have
$\dim\mathcal{T}_xV\ge \dim_xV$ (see, e.g., \cite[page
94]{Shafarevich94}). The point $x$ is {\em regular} if
$\dim\mathcal{T}_xV=\dim_xV$; otherwise, $x$ is called {\em
	singular}. The set of singular points of $V$ is the {\em singular
	locus} of $V$; it is a closed $\K$--subvariety of
$V$. A variety is called {\em nonsingular} if its singular locus is
empty. For projective varieties, the concepts of tangent space,
regular and singular point can be defined by considering an affine
neighborhood of the point under consideration.
%
%
\subsection{Rational points}
Let $\Pp^r(\fq)$ be the $r$--dimensional projective space over $\fq$
and $\A^r(\fq)$ the $r$--dimensional $\fq$--vector space $\fq^r$.
For a projective variety $V\subset\Pp^r$ or an affine variety
$V\subset\A^r$, we denote by $V(\fq)$ the set of $\fq$--rational
points of $V$, namely $V(\fq):=V\cap \Pp^r(\fq)$ in the projective
case and $V(\fq):=V\cap \A^r(\fq)$ in the affine case. For an affine
variety $V$ of dimension $n$ and degree $\delta$, we have the
following bound (see, e.g., \cite[Lemma 2.1]{CaMa06}):
\begin{equation}\label{eq: upper bound -- affine gral}
|V(\fq)|\leq \delta\, q^n.
\end{equation}
On the other hand, if $V$ is a projective variety of dimension $n$
and degree $\delta$, then we have the following bound (see
\cite[Proposition 12.1]{GhLa02a} or \cite[Proposition 3.1]{CaMa07};
see \cite{LaRo15} for more precise upper bounds):
\begin{equation*}\label{eq: upper bound -- projective gral}
|V(\fq)|\leq \delta\, p_n,
\end{equation*}
where $p_n:=q^n+q^{n-1}+\cdots+q+1=|\Pp^n(\fq)|$.
%
%
\subsection{Complete intersections}\label{subsec: complete intersections}
A \emph{set–theoretic complete intersection} is an affine $\K$-variety $ V (F_1,\ldots,F_m)\subset\A^r$ or a projective $\K$-variety $V (F_1,\ldots,F_m)\subset \Pp^r$ defined by $m \leq r$ polynomials $F_1, \ldots,F_m \in \K[X_1, \ldots, X_r]$ or homogeneous polynomials $F_1, \ldots,F_m$ in $\mathbb{K}[X_0,\ldots,X_r]$, which is of pure dimension $r-m$. If in addition $(F_1, \ldots,F_m)$ is a radical ideal of $\K[X_1, \ldots, X_r]$ or $\K[X_1, \ldots,X_r]$, then we say that $V (F_1,\ldots,F_m)$ is an \emph{ideal–theoretic complete intersection}. 
Elements $F_1,\ldots, F_m$ in $\mathbb{K}[X_1,\ldots,X_r]$ or
$\mathbb{K}[X_0,\ldots,X_r]$ form a \emph{regular sequence} if the ideal $(F_1, \ldots, F_m)$ they define in $\K[X_1, \ldots,X_r]$ or $K[X_0, \ldots, X_r]$ is a proper ideal, $F_1$
is nonzero and, for $2 \leq i \leq m$,  $F_i$ is  neither zero nor a zero divisor in the quotient
ring $\mathbb{K}[X_1,\ldots,X_r]/ (F_1,\ldots,F_{i-1})$ or
$\mathbb{K}[X_0,\ldots,X_r]/ (F_1,\ldots,F_{i-1})$.  In such a case, the (affine or projective) variety $ V (F_1,\ldots,F_m)$ they define is a set–theoretic complete intersection.

 We shall frequently use the following criterion to
prove that a variety is a complete intersection (see, e.g.,
\cite[Theorem 18.15]{Eisenbud95}).
\begin{theorem}\label{theorem: eisenbud 18.15}
	Let $F_1,\ldots,F_m\in\mathbb{K}[X_1,\ldots,X_r]$ be polynomials
	which form a regular sequence and let
	$V:=V(F_1,\ldots,F_m)\subset\A^r$. Denote by
	$(\partial\bfs{F}/\partial \bfs{X})$ the Jacobian matrix of
	$F_1,\ldots,F_m$ with respect to $\bfs{X}:=(X_1,\ldots,X_r)$. If the subvariety
	of $V$ defined by the set of common zeros of the maximal minors of
	$(\partial\bfs{F}/\partial \bfs{X})$ has codimension at least one in
	$V$, then $F_1,\ldots,F_m$ define a radical ideal. In particular,
	$V$ is a complete intersection.
\end{theorem}

If $V\subset\Pp^r$ is a complete intersection defined over $\K$ of
dimension $r-m$, and $F_1 ,\ldots, F_m$ is a system of homogeneous
generators of $I(V)$, the degrees $d_1,\ldots, d_m$ depend only on
$V$ and not on the system of generators. Arranging the $d_i$ in such
a way that $d_1\geq d_2 \geq \cdots \geq d_m$, we call $(d_1,\ldots,
d_m)$ the {\em multidegree} of
$V$. In this case, a stronger version of 
\eqref{eq: Bezout} holds, called the {\em B\'ezout theorem} (see,
e.g., \cite[Theorem 18.3]{Harris92}):
\begin{equation}\label{eq: Bezout theorem}
\deg V=d_1\cdots d_m.
\end{equation}
A complete intersection $V$ is called {\em normal} if it is {\em
	regular in codimension 1}, that is, the singular locus
$\mathrm{Sing}(V)$ of $V$ has codimension at least $2$ in $V$,
namely $\dim V-\dim \mathrm{Sing}(V)\ge 2$ (actually, normality is a
general notion that agrees on complete intersections with the one we
define here). A fundamental result for projective complete
intersections is the Hartshorne connectedness theorem (see, e.g.,
\cite[Theorem VI.4.2]{Kunz85}): if $V\subset\Pp^r$ is a complete
intersection defined over $\K$ and $W\subset V$ is any
$\K$--subvariety of codimension at least 2, then $V\setminus W$ is
connected in the Zariski topology of $\Pp^r$ over $\K$. Applying the
Hartshorne connectedness theorem with $W:=\mathrm{Sing}(V)$, one
deduces the following result.
\begin{theorem}\label{theorem: normal complete int implies irred}
	If $V\subset\Pp^r$ is a normal complete intersection, then $V$ is
	absolutely irreducible.
\end{theorem}
%
%
%
%
%
%
\section{Systems of diagonal equations}

Let  $k,m, d_1,\ldots,d_m$ be positive integers such that $m \leq \frac{k-1}{2}$, $2\leq d_1<d_2<\ldots <d_m$. We can assume that $\mathrm{char}(\fq):=p$ does  not divide $d_i$ for $1\leq i \leq m$. Indeed, suppose that there exists $i$ such that $p^r t=d_i$ and $p$ does not divide $t$. Then we can replace the equation $X_1^{d_i}+X_2^{d_i}+\cdots + X_k^{d_i} = b_i$ by the equation $X_1^{t}+X_2^{t}+\cdots + X_k^{t}=b_i'$ with $b_i'\in \fq$.

We consider the following system of $m$  diagonal equations with $k$ unknowns
\begin{equation}\label{eq: system}\left\{ \begin{array}{lcc}
             X_1^{d_1}+X_2^{d_1}+\cdots + X_k^{d_1}=b_1 \\
            \hspace{50px}\vdots\\ X_1^{d_m}+X_2^{d_m}+\cdots + X_k^{d_m}=b_m \\
             \end{array}
   \right.    
 \end{equation}
where  $b_i\in \mathbb{F}_q$, $1\leq i\leq m$.

Let $N$ denote the number of $\fq$--rational solutions of \eqref{eq: system}. The purpose of this section is to give an estimate on the number $N$. 
 To do this, we consider the following polynomials $f_j\in \fq[X_1, \ldots,X_k]$
 \begin{equation}\label{def: f'i}
 f_j:=X_1^{d_j}+X_2^{d_j}+\cdots+X_k^{d_j}-b_j,\,\,1 \leq j \leq m.
 \end{equation}
Let $V:=V(f_1, \ldots,f_m)\subset \A^k$ be the $\fq$--affine variety defined by $f_1, \ldots,f_m$. We shall study some facts concerning the geometry of $V$. 
\begin{lemma} \label{lemma: codimension del lugar singula de V}
 Let $m\leq \frac{k-1}{2}$, the codimension of the singular locus of $V$ is at least two and the ideal $\left(f_1,\ldots,f_m\right)$ is radical.

\end{lemma}

\begin{proof}

We consider the following set
\[
\left\{\boldsymbol{x}\in V:\mathrm{rank}\left(\frac{\partial f}{\partial \boldsymbol{X}}\right)\Big |_{\boldsymbol{x}} <m\right\}=\bigcup_{i=0}^{m-1} C_i,
\]

where  $\left(\frac{\partial f}{\partial \boldsymbol{X}}\right)$ is the Jacobian $(m\times k)$--matrix of $f_1,\ldots,f_m$ with respect to $X_1,\ldots,X_k$, and     $C_i=\left\{\boldsymbol{x}\in V: \mathrm{rank}\left(\frac{\partial f}{\partial \boldsymbol{X}}\right)\Big |_{\boldsymbol{x}}=i\right\}.$

Given an $\boldsymbol{x}\in C_i$, there is a minor of $\left(\frac{\partial f}{\partial \boldsymbol{X}}\right)$  of size $i\times i$ which is not zero. 
More precisely, without loss of generality, we can suppose that
 the $i\times i$--submatrix $B_i(\boldsymbol{X})$ which consists of the first $i$ rows and the first $i$
columns of $\left(\frac{\partial f}{\partial \boldsymbol{X}}\right)$, 
satisfies that $\det(B_i(\boldsymbol{x}))\neq 0$.
Then $x_k\neq 0$ for $k=1,\ldots, i$ and $x_k\neq x_j$ if $k\neq j$.

On the other hand, we consider the $(i+1)\times(i+1)$--matrices $A_j(\boldsymbol{X})$ define by
\[A_j(\boldsymbol{X})=
\begin{pmatrix}
    &  & & \vline & d_1X_j^{d_{1-1}}\\
     & B_i(\boldsymbol{X}) & & \vline & \vdots\\
    & & & \vline  & d_iX_j^{d_{i-1}}\\
    \hline
    d_{i+1}X_1^{d_{i+1-1}} & \cdots & d_{i+1}X_i^{d_{i+1-1}} & \vline  & d_{i+1}X_j^{d_{i+1-1}}
\end{pmatrix}
\]
which satisfy that $\det(A_j(\boldsymbol{x}))=0$ for $j=i+1,\ldots,k$. From \cite[Theorem 2.1]{Marchi} we obtain

\[\det(A_j(\boldsymbol{x}))=d_1\cdots d_{i+1}VDM(x_1\ldots x_i)\cdot \prod_{l=1}^ix_l^{d_1-1}\cdot \prod_{l=1}^i(x_j-x_l)\cdot P(x_j)\]

for $j=i+1,\ldots,k$, with $P\in \fq[T]$ and $VDM(x_1\ldots x_i) $ is  the classical  Vandermonde determinant in $x_1,\ldots,x_i$.


Hence,  $Q_j(\boldsymbol{x})=0$, $j=i+1,\ldots,k$, where
$Q_j(\boldsymbol{X}):=\prod_{l=1}^i(X_j-X_l)\cdot P(X_j)$.

Taking $X_k>\cdots > X_1$ for a graded lexicographic order it turns out that $LT(Q_j(\boldsymbol{X}))=X_j^{i+M}$ for $j=i+1,\ldots,k$, where $M=\deg(P)$.
Thus $LT(Q_j(\boldsymbol{X}))$, $i+1\leq j\leq k,$ are relatively prime and therefore, they form a
Gröbner basis of the ideal $J$ that they generate (see, e.g., \cite[\S 2.9, Proposition 4]{CLO92}). Then
the initial of the ideal $J$ is generated by $LT(Q_{i+1}(\boldsymbol{X})),\ldots,LT(Q_k(\boldsymbol{X}))$, which form a regular sequence
of $\fq[X_1,\ldots,X_k]$. Therefore, by \cite[Proposition 15.15]{Eisenbud95}, the polynomials $Q_{i+1}(\boldsymbol{X}),\ldots,Q_k(\boldsymbol{X})$ form
a regular sequence of $\fq[X_1,\ldots,X_k]$. We conclude that $V(Q_{i+1}(\boldsymbol{X}),\ldots,Q_{k}(\boldsymbol{X}))$ is a set  theoretic complete intersection of $\A^k$ of dimension $i$.

Now, by the definition of $Q_j(\boldsymbol{X})$, $i+1\leq j\leq k$, $C_i\subset V(Q_{i+1}(\boldsymbol{X}),\ldots,Q_k(\boldsymbol{X}))$ hold which implies that  $\dim C_i\leq i$.










We can conclude that \[\mathrm{dim}\left(\left\{x\in V:\mathrm{rank}\left(\frac{\partial f}{\partial \boldsymbol{X}}\right)\Big |_{\boldsymbol{x}}<m\right\}\right)\leq m-1,\]
in particular $\mathrm{Sing(V)}$ has dimension at most $m-1$.


On the other hand, by \cite[Chapter 5, Theorem 3.4]{Kunz85} we have that $\dim(V)\geq k-m$, and taking into account that it turns out that  $m\leq \frac{k-1}{2}$
we conclude that  $\mathrm{codim}(\mathrm{Sing}(V))\geq 2$.

Aditionally, as $\{Q_{i+1}(\boldsymbol{X}),\ldots,Q_k(\boldsymbol{X})\}$ are a regular sequence we have that by \cite[Theorem 18.15]{Eisenbud95} $\left(f_1,\ldots,f_s\right)$ is a radical ideal which concludes the proof.

\end{proof}

\begin{theorem} \label{teo: V es interseccion completa}
If $m\leq\frac{k-1}{2}$ then $V$ is a complete intersection of pure dimension $k-m$ and $\deg(V)\leq d_1\cdots d_m$.
\end{theorem}

\begin{proof}
Let $C$ be an irreducible component of $V$. By \cite[Chapter 5, Theorem 3.4]{Kunz85} we have that $\dim(C)\geq k-m$. Let 
$J:=\left\{\boldsymbol{x}\in C:\mathrm{rank}\left(\frac{\partial f}{\partial \boldsymbol{X}}\right)\Big|_{\boldsymbol{x}}<m\right\}$. By the proof of previous lemma, we 
have that $\dim(J)\leq m-1$. Therefore, since $m\leq \frac{k-1}{2}$, the following inequality holds:

\begin{align*}
    \dim(C)-\dim(J)\geq k-m-(m-1)\geq 1,
\end{align*}

which implies that there exists some $\boldsymbol{p}\in J$ such that $\mathrm{rank}\left(\frac{\partial f}{\partial \boldsymbol{X}}\right)\Big|_{\boldsymbol{p}}=m$.

Now, as $T_{\boldsymbol{p}} V(f_1,\ldots,f_m)\subset \mathrm{ker}\left(\frac{\partial f}{\partial \boldsymbol{X}}\Big|_{\boldsymbol{p}}\right)$ and 
$ \dim\left(\mathrm{ker}\left(\left(\frac{\partial f}{\partial \boldsymbol{X}}\Big|_{\boldsymbol{p}}\right)\right)\right)=k-m,
$
we can conclude that $\dim\left(T_{\boldsymbol{p}} V(f_1,\ldots,f_m)\right)\leq k-m$.
It follows that $\dim C\leq \dim_p(V)  \leq k-m.$



It follows that  $V$ is of pure dimension $k-m$ which implies that $V$ is a set theoretic complete intersection. Since we proved that $\left(f_1,\ldots,f_m\right)$ is a radical ideal we have that $V$ is a complete intersection and by Bezout inequality \eqref{eq: Bezout} $\deg(V)\leq d_1\cdots d_m$.
\end{proof}
\subsection{The geometry of the projective closure}\label{geo proyectiva}
Consider the embedding of $\A^k$ into the projective space $\Pp^k$
which assigns to any $\bfs{x}:=(x_1,\ldots, x_k)\in\A^k$ the point
$(1:x_1:\dots:x_k)\in\Pp^k$. Then the closure
$\mathrm{pcl}(V)\subset\Pp^k$ of the image of $V$ under this
embedding in the Zariski topology of $\Pp^k$ is called the
projective closure of $V$. The points of $\mathrm{pcl}(V)$ lying
in the hyperplane $\{X_0=0\}$ are called the points of
$\mathrm{pcl}(V)$ at infinity.

It is well--known that $\mathrm{pcl} (V)$ is the $\fq$--variety of
$\mathbb{P}^k$ defined by the homogenization
$F^h\in\fq[X_0,\ldots,X_k]$ of each polynomial $F$ belonging to the
ideal $(f_1\klk f_m)\subset\fq[X_1,\ldots,X_k]$ (see, e.g.,
\cite[\S I.5, Exercise 6]{Kunz85}). Denote by $(f_1\klk f_m)^h$
the ideal generated by all the polynomials $F^h$ with $F\in
(f_1\klk f_m)$. Since $(f_1\klk f_m)$ is radical it turns
out that $(f_1\klk f_m)^h$ is also a radical ideal (see, e.g.,
\cite[\S I.5, Exercise 6]{Kunz85}). Furthermore, $\mathrm{pcl}
(V)$ has pure dimension $k-m$ (see, e.g., \cite[Propositions
I.5.17 and II.4.1]{Kunz85}) and degree equal to $\deg V$ (see, e.g.,
\cite[Proposition 1.11]{CaGaHe91}).
Now we discuss the behaviour of $\mathrm{pcl} (V)$ at infinity. Recall that $V=V({f}_1 \klk {f}_m) \subset \A^k$, where $f_j:=X_{1}^{d_{j}}+ \cdots + X_k^{d_{j}}-b_j,$ $ 1\leq j\leq m$.
Hence, the homogenization of each ${f}_j$ is the following polynomial of $\fq[X_0 \klk X_k]:$ $$f_j^{ h}:=X_{1}^{d_{j}}+ \cdots + X_k^{d_{j}}-b_jX_0^{d_j}, \, \,   1 \leq j \leq m.$$
Furthermore,  $f_j^{ h}\Big|_{X_0=0}=X_{1}^{d_{j}}+ \cdots + X_k^{d_{j}}=f_j+b_j$. Observe that, since $f_i+b_i$, $1\leq i \leq m$, are homogeneous polynomials then, $V(f_1+b_1,\ldots,f_m+b_m) \subset \Pp^{k-1}$ is a projective variety of dimension at least $k-m-1$ and we have that  $V^{\infty}\subset V(f_1+b_1,\ldots,f_m+b_m)$, where $V^{\infty}:=\mathrm{pcl}(V) \cap \{X_0=0\}$.

\begin{remark}\label{props V(fi+bi)}
    By Lemma \ref{lemma: codimension del lugar singula de V} and Theorem \ref{teo: V es interseccion completa} $V(f_1+b_1,\ldots,f_m+b_m)\subset \A^k$ is an affine cone of pure dimension $k-m$ and its singular locus has dimension at most $m-1$. It follows that the projective variety $V(f_1+b_1,\ldots,f_m+b_m)$ is of pure dimension $k-m-1$  and its singular locus has dimension at most $m-2$ and degree at most $d_1\cdots d_m$.
\end{remark}

\begin{lemma}\label{dim sing(pcl(V)) leq m-2}
   Let $m\leq \frac{k-1}{2}$,  $\mathrm{pcl} (V)$ has singular locus at infinity of dimension at most $m-2$.
\end{lemma}

\begin{proof}
    Let $\Sigma^{\infty}\subset \Pp^k$ be the singular locus of $\mathrm{pcl}(V)$ at infinity.

    Consider $\bfs{x}=(0:x_1:\cdots:x_k)\in \Sigma^{\infty}$, as the polynomials $f_i^h$ vanish in $\mathrm{pcl} (V)$,  we have that $0=f_i^h(\bfs{x})=(f_i+b_i)(x_1,\ldots,x_k)$.

    Let $\left(\frac{\partial \bfs{f+b}}{\partial\bfs{X}}\right)$ be the Jacobian matrix of the polynomials  
$f_i+b_i$, $1\leq i \leq m$, with respect to $X_1,\ldots, X_k.$
    It holds that $\rank\left(\frac{\partial \bfs{f+b}}{\partial\bfs{X}}\right)\Big|_{\bfs{x}}<m$. Indeed, if  $\rank\left(\frac{\partial\bfs{f+b}}{\partial\bfs{X} }\right)\Big|_{\bfs{x}}=m$ then 
    
\begin{equation*}   k=\dim\left(\ker\left(\frac{\partial \bfs{f+b}}{\partial \bfs{X}}\Big|_{\bfs{x}}\right)\right)+\rank\left(\frac{\partial\bfs{f+b} }{\partial\bfs{X}}\Big|_{\bfs{x}}\right)\\
   \end{equation*}    

    Since $T_{\bfs{x}}\left(\mathrm{pcl} (V)\right)\subset \ker\left(\frac{\partial \bfs{f+b}}{\partial \bfs{X}}\Big|_{\bfs{x}}\right)$ it follows that $\dim(T_{\bfs{x}}(\mathrm{pcl} (V)))\leq k-m$  which implies that $\bfs{x}$ is regular point of  $\mathrm{pcl} (V)$ which contradicts the hypothesis on $\bfs{x}$.

    By the proof of Lemma \ref{lemma: codimension del lugar singula de V} the set  of points  of $V$ where $\rank\left(\frac{\partial \bfs{f+b}}{\partial \bfs{X}}\Big|_{\bfs{x}}\right)<m$ has dimension at most $m-1$ and therefore its dimension as a projective variety is at most $m-2$.    
\end{proof}
Now we are able to completely characterize the behavior of $\mathrm{pcl} (V)$ at infinity.
\begin{theorem}\label{teo: props de la clausura en el infinito} 
Let $m\leq \frac{k-1}{2}$,
    $\mathrm{pcl}(V)\cap \{X_0=0\}$ is a absolutely irreducible complete intersection  of dimension $k-m-1$, degree $d_1\cdots  d_m$ and its singular locus has dimension at most $m-2$.
\end{theorem}

\begin{proof}
    Recall that $\mathrm{pcl}(V)\cap \{X_0=0\}\subset V(f_1+b_1,\ldots,f_m+b_m) \subset \Pp^{k-1}$. By Remark \ref{props V(fi+bi)} we have that $V(f_1+b_1,\ldots,f_m+b_m)$ is of pure dimension $k-m-1$, its singular locus has dimension at most $m-2$ and its degree is at most $d_1\cdots d_m$. Then $V(f_1+b_1,\ldots,f_m+b_m)$ is a set theoretic complete intersection and the codimension of its singular locus is at least $2$. On the other hand, since $f_1+b_1,\ldots,f_m+b_m$ are homogeneous polynomials and define a variety of dimension $k-m-1$ then they form a regular sequence. Hence,  from Theorems \ref{theorem: eisenbud 18.15} and \ref{theorem: normal complete int implies irred},  $V(f_1+b_1,\ldots,f_m+b_m)$ is an absolutely irreducible complete intersection. Taking into account that $\mathrm{pcl}(V)\cap \{X_0=0\}$ has dimension $k-m-1$, we deduce that $V(f_1+b_1,\ldots,f_m+b_m)=\mathrm{pcl}(V)\cap \{X_0=0\}$.
    Finally, by Theorem \ref{theorem: eisenbud 18.15}, we obtain that $\mathrm{deg}(V(f_1+b_1,\ldots,f_m+b_m))$ is equal to $d_1\cdots d_m$.
\end{proof}

\begin{theorem} \label{teo: pcl V interseccion completa}
  Let $m\leq \frac{k-1}{2}$, the projective variety $\mathrm{pcl} (V)\subset \Pp^k$ is an absolutely irreducible complete intersection  of dimension $k-m$, degree $d_1\cdots d_m$ and the dimension of its singular locus is at most $m-1$.
\end{theorem}

\begin{proof}
    $\mathrm{pcl}(V)$ is a projective variety of pure dimension $k-m$ and degree $d_1\cdots d_m$. By Lemma \ref{lemma: codimension del lugar singula de V}, the set of singular points  of $\mathrm{pcl}(V)\cap \{X_0\neq 0\}$ has dimension at most $m-1$. On the other hand, by Lemma \ref{dim sing(pcl(V)) leq m-2} the set of singular points of $\mathrm{pcl}(V)\cap \{X_0= 0\}$ has dimension at most $m-2$. Therefore the dimension of the singular locus of $\mathrm{pcl}(V)$ is at most $m-1$.

On the other hand, we have that $\mathrm{pcl}(V)\subset V(f_1^h,\ldots,f_m^h)$,
$V(f_1^h,\ldots,f_m^h)\cap\{X_0\neq 0\}\subset V(f_1,\ldots,f_m)$ and
  $V(f_1^h,\ldots,f_m^h)\cap\{X_0=0\}\subset V(f_1+b_1,\ldots,f_m+b_m)$.

    By Theorem \ref{teo: V es interseccion completa}, $V$ is of pure dimension $k-m$  and by Theorem \ref{teo: props de la clausura en el infinito} $V(f_1+b_1,\ldots,f_m+b_m)$ is of pure dimension $k-m-1$. It follows that $V(f_1^h,\ldots,f_m^h)$  has dimension at most $k-m$.

    As $V(f_1^h,\ldots,f_m^h)$ is defined by $m$ polynomials we conclude that $V(f_1^h,\ldots,f_m^h)$  is a set theoretic complete intersection of dimension $k-m$. 

    Observe that the set of singular points of $V(f_1^h,\ldots,f_m^h)$ is a subset of the set of  points of $\Pp^k$ for which the matrix $\left(\frac{\partial \bfs{f^h}}{\partial \bfs{X}}\right)$ has not full rank. By Lemma \ref{lemma: codimension del lugar singula de V} the set of points of $V(f_1^h,\ldots,f_m^h)\cap \{X_0\neq 0\}$ where the matrix $\left(\frac{\partial \bfs{f^h}}{\partial \bfs{X}}\right)$ has not full rank, has dimension at most $m-1$. Analogously, the set of points of $V(f_1^h,\ldots,f_m^h) \cap \{X_0=0\}$ where the matrix has not full rank, has dimension at most  $m-2$. Hence, it follows that the codimension of the singular locus of $V(f_1^h,\ldots,f_m^h)$ is at least $2$ and $(f_1^h,\ldots,f_m^h)$ is a radical ideal. Finally, we conclude that $V(f_1^h,\ldots,f_m^h)$ is absolutely irreducible then, $\mathrm{pcl}(V)=V(f_1^h,\ldots,f_m^h)$ of degree $d_1\cdots d_m$.
\end{proof}

\subsection{Estimates on the number of $\fq$--rational solutions of  systems of diagonal equations}
Let  $k,m, d_1,\ldots,d_m$ be positive integers such that $m \leq \frac{k-1}{2}$,  and $2 \leq d_1<d_2< \ldots <d_m$. 

In what follows, we shall use an estimate on the number of
$\fq$--rational points of a projective complete intersection due to S. Ghorpade
and G. Lachaud (\cite{GhLa02a}; see also \cite{GhLa02}). In \cite[Theorem 6.1]{GhLa02a}, the authors prove that,
for an irreducible $\fq$--complete intersection $V\subset \Pp^N$
of dimension $r$, multidegree ${\bfs{d}}=(d_1,\ldots,d_{N-r})$ and  singular locus of dimension at most $s$ with $0\leq s\leq r-1$, the number $|V(\fq)|$ of $\fq$--rational points of $V$ satisfies the estimate:
\begin{equation}\label{eq: estimacion de Ghorpade Lachaud}
\big||V(\fq)|-p_r\big|\leq b_{r-s-1}'(N-s-1,{\bfs{d}})\, \, q^{\frac{r+s+1}{2}}+C_s(V)q^{\frac{r+s}{2}},
\end{equation}
where $p_r:=q^r+q^{r-1}+\cdots+1$,  $b_{r-s-1}'(N-s-1,{\bfs{d}})$ is the  $(r-s-1)$--th primitive Betti of a nonsingular complete intersection in $\Pp^N$
of dimension $r-s-1$ and  multidegree ${\bfs{d}}$,
and $C_s(V):=\sum_{i=r}^{r+s}b_{i,\ell}(V)+\varepsilon_i$, where
$b_{i,\ell}(V)$ denotes the $i$--th $\ell$--adic Betti number of $V$ for a prime $\ell$ different from $p:=\mathrm{char}(\fq)$ and
$\varepsilon_i:=1$ for even $i$ and $\varepsilon_i:=0$ for odd $i$.
From \cite[Proposition 4.2]{GhLa02a}
\begin{equation}\label{cota de Betti}
b_{r-s-1}'(N-s-1,{\bfs{d}})\leq \binom{N-s}{r-s-1}\cdot (d+1)^{N-s-1},
\end{equation}
where $d:=\max\{d_1,\ldots,d_{N-r}\}$.
On the other hand, from \cite[Theorem 6.1]{GhLa02a}, we have that
\begin{equation}\label{constante C}
C_s(V)\leq 9\cdot 2^{N-r}\cdot ((N-r)d+3)^{N+1}.
\end{equation}

Denote by  $\mathrm{pcl}(V)(\fq)$ the set $\fq$--rational points of $\mathrm{pcl}(V)$. From Theorem \ref{teo: pcl V interseccion completa}, the estimate \eqref{eq: estimacion de Ghorpade Lachaud}, the bounds \eqref{cota de Betti} and \eqref{constante C}, we deduce that
\begin{equation}\label{estimation pcl V}
    \big||\mathrm{pcl}(V(\mathbb{F}_q))|-p_{k-m}\big|\leq{k-m+1\choose k-2m}{(d_m+1)}^{k-m}q^{\frac{k}{2}}+9\cdot 2^{m}(md_m+3)^{k+1}q^{\frac{k-1}{2}}.
\end{equation} 

Now we estimate the number of $\fq$-rational points of $V^{\infty}\subset \Pp^{k-1}$. From Theorem \ref{teo: props de la clausura en el infinito}, the bounds \eqref{cota de Betti} and \eqref{constante C},  we obtain


\begin{equation} \label{estimacion pcl(V) infinito}
    \big||\mathrm{V ^{\infty}}(\fq)|-p_{k-m-1}\big|\leq{k-m+1\choose k-2m}{(d_m+1)}^{k-m}q^{\frac{k-2}{2}}+9\cdot2^{m}(md_m+3)^{k}q^{\frac{k-3}{2}}.
\end{equation}
With this results proven we can now state our main theorem.
\begin{theorem}\label{teo: estimacion ptos racionales V}
Let  $k,m, d_1,\ldots,d_m$ be positive integers such that $m \leq \frac{k-1}{2}$,  and $2 \leq d_1<d_2< \ldots <d_m$. Suppose that $\mathrm{char}(\fq):=p$ does not divide $d_i$ for $1 \leq i \leq m$. Let $|V(\fq)|$ the number of $\fq$--rational points of $V:=V(f_1, \ldots,f_m) \subset \A^k$, where $f_i, (1 \leq i \leq m)$ are defined as \eqref{def: f'i}. We have
\begin{equation*}
\left ||V(\mathbb{F}_q)|-q^{k-m}\right|\leq 3^3\cdot 2^{m-1}(3+d_mm)^{k+1}\cdot q^{\frac{k}{2}}.
\end{equation*}
\end{theorem}

\begin{proof}
From \eqref{estimation pcl V} and \eqref{estimacion pcl(V) infinito} and the fact that $\mathrm{pcl}(V)(\fq) \setminus V^{\infty}(\fq)=V(\fq)$, we obtain that







\begin{align*}
    \big||V(\fq)|-q^{k-m}\big|&\leq{k-m+1\choose k-2m}{(d_m+1)}^{k-m}q^{\frac{k}{2}}+9\cdot 2^{m}(md_m+3)^{k+1}q^{\frac{k-1}{2}}\\
    &+{k-m+1\choose k-2m}{(d_m+1)}^{k-m}q^{\frac{k-2}{2}}+9\cdot2^{m}(md_m+3)^{k}q^{\frac{k-3}{2}}.
\end{align*}
Taking into account that $\binom{a}{b} \leq 2^a$, we obtain

\begin{align*}
    \big||V(\fq)|-q^{k-m}\big|&\leq2^{k-m+1}{(d_m+1)}^{k-m}q^{\frac{k}{2}}+9\cdot 2^{m}(md_m+3)^{k+1}q^{\frac{k-1}{2}}\\
    &+2^{k-m+1}{(d_m+1)}^{k-m}q^{\frac{k-2}{2}}+9\cdot2^{m}(md_m+3)^{k}q^{\frac{k-3}{2}}\\
    &\leq 2^{k-m+1}(d_m+1)^{k-m}q^{\frac{k-2}{2}}(1+q)+9\cdot2^{m}(md_m+3)^{k+1}q^{\frac{k-3}{2}}(1+q)\\
    &= 2 (2d_m+2)^{k-m} q^{\frac{k-2}{2}}(1+q)+9\cdot2^{m}(md_m+3)^{k+1}q^{\frac{k-3}{2}}(1+q)\\
    &\leq 9\cdot2^{m}(md_m+3)^{k+1}q^{\frac{k-3}{2}}(1+\frac{q^{\frac{1}{2}}}{9 \cdot 2^{m-1}})(1+q)\\
    &\leq 9\cdot2^{m}(md_m+3)^{k+1}q^{\frac{k-3}{2}}\frac{3}{2}q^{\frac{3}{2}}\\
    &\leq 3^3\cdot2^{m-1}(md_m+3)^{k+1}q^{\frac{k}{2}}.\\
\end{align*}

Which concludes our main result.

\end{proof}
\begin{remark}
    Observe that if $m=1$ we have the following diagonal equation: $X_1^d+\cdots+X_k^d=b$. Then according to Theorem \ref{teo: estimacion ptos racionales V} we obtained that $\big||V(\fq)|-q^{k-1}\big|\leq 27(d+3)^{k+1}q^{\frac{k}{2}}$.

    In \cite[Remark 4.4]{PP20} we provided an estimate  which has an error term of the same order, indeed we proved that $\big||V(\fq)|-q^{k-1}\big|\leq2(d-1)^k q^{\frac{k}{2}}.$
\end{remark}
From Theorem \ref{teo: estimacion ptos racionales V} we obtain the following existence result.
\begin{proposition}
    Let $N$ be the number of $\fq$--rational solutions of the system \eqref{eq: system}, $k\geq 5$ and $m\leq \frac{k-1}{2}$. If $q> (\frac{7}{2}md_m)^{\frac{2k+2}{k-2m}}$ then $N>0$. 
\end{proposition}
\begin{proof}
    From Theorem \ref{teo: estimacion ptos racionales V} we have that $N\geq q^{k-m}-27\cdot2^{m-1}(3+md_m)^{k+1}q^{\frac{k}{2}}$. Then, the statement of the proposition follows from the condition  $ q^{k-m}-27\cdot2^{m-1}(3+md_m)^{k+1}q^{\frac{k}{2}}>0.$
\end{proof}
\subsection{Homogeneous case}
Assume that $p$ does not divide $k!$. We shall consider the system \eqref{eq: system} with $\bfs{b}=\bfs{0}$ and  $d_1=1, \ldots, d_m=m$. In this case, we can obtain a better estimate on the number of $\fq$--rational solutions of the system. Observe that every $\fq$--rational solution of the homogeneous systems corresponds to a $\fq$--rational point of the affine variety $V_0=V(\Pi_1,\ldots,\Pi_m)\subset \A^k$ and {\it{vice versa}}. Indeed, we define 
\[
P_j(x_1,\ldots,x_k)=\sum_{i=1}^kx_i^j
\]
Then, by Newton's identity we have that:
\begin{align*}
    P_k&=(-1)^{k-1}k\cdot\Pi_k+\sum_{i=1}^{k-1}\Pi_{k-i}P_i
\end{align*}


where, $\Pi_i\in \mathbb{F}_q[X_1,\ldots,X_k]$ represents the $i$-th symmetric elementary polynomial, $1\leq i\leq k$.

Given that, it is easy to see that any solution of the homogeneous system is also a solution of the system 

\begin{equation*}\label{eq: symetric elemental system}\left\{ \begin{array}{lcc}
             \Pi_1(x_1,\ldots,x_k)=0 \\
             \Pi_2(x_1,\ldots,x_k)=0 \\
            \vdots\\ 
            \Pi_m(x_1,\ldots,x_k)=0 \\
             \end{array}
   \right.    
 \end{equation*}
and \textit{vice versa}.

We shall apply the results of the article \cite{GiMaPePr2023}. In this sense, we consider the polynomials $G_i=Y_i$, $1\leq i\leq m$ and the linear affine variety $W=V(Y_1,\ldots,Y_m)\subset \A^k$ of dimension $k-m$. In this context, the variety $V_0$ is given by the polynomials $F_i:=G_i(\Pi_1,\ldots,\Pi_m)$, $1\leq i \leq m$. It is easy to check that $W$ satisfies the hypothesis $({\sf A}_1)$ and
$({\sf A}_2)$ of \cite{GiMaPePr2023}.
From \cite[Theorem 4.3, Lemma 4.4, Theorem 4.5]{GiMaPePr2023} we have that  if $k-m \geq 2$, $\mathrm{pcl}(V_0)$ is a complete intersection of dimension $k-m$, of degree equals to $m!$ and it is regular in codimension $k-m-1$ . Furthermore, $\mathrm{pcl}(V_0)$ has not singular points in the infinity hyperplane. Hence, by \cite[Corollary 4.6]{GiMaPePr2023}, $\mathrm{pcl}(V_0)^{\infty}:=\mathrm{pcl}(V_0)\cap \{X_0=0\}$ is a nonsingular complete intersection of dimension $k-m-1$ and degree $m!$. From  \eqref{eq: estimacion de Ghorpade Lachaud}, we obtain the following estimate on the number of $\fq$--rational points of $\mathrm{pcl}(V_0)$:
\begin{equation}\label{eq: GL caso homogeneo}
\big||\mathrm{pcl}(V_0)(\fq)|-p_{k-m}\big|\le
\binom{k+1}{k-m}(m+1)^k q^{\frac{k-m+1}{2}} + 9 \cdot 2^m (m^2+3)^{k+1} q^{\frac{k-m}{2}}.
\end{equation}
On the other hand, since $\mathrm{pcl}(V_0)\cap \{X_0=0\}$ is a nonsingular complete intersection, we can apply the well known estimate due P. Deligne (see, \cite{De74})
\begin{equation}\label{eq: estimacion Deligne}
\big||\mathrm{pcl}(V_0)^{\infty}(\fq)|-p_{k-m-1}\big|\le
\binom{k}{k-m-1} (m+1)^{k-1} q^{\frac{m-k-1}{2}}.
\end{equation}
Finally, from \eqref{eq: GL caso homogeneo}, \eqref{eq: estimacion Deligne} and taking into account that $\binom{a}{b} \leq 2^a$, we deduce the following estimate on the number of $\fq$--rational points of $V_0$:
\begin{align} \label{systems homogeneous}
\big||V_0(\fq)|-q^{k-m}\big| &\leq 2^{k+1} (m+1)^k q^{\frac{k-m+1}{2}}+ 9\cdot  2^m(m^2+3)^{k+1} q^{\frac{k-m}{2}}+2^k (m+1)^{k-1} q^{\frac{k-m-1}{2}}\\ \notag
& \leq (m^2+3)^{k+1} q^{\frac{k-m-1}{2}} (2q+3^2 2^m q^{1/2}+2) \\ \notag
& \leq 3^3 2^{m-1} (m^2+3)^{k+1} q^{\frac{k-m+1}{2}}.
\end{align}





\section{Application to the subset sum problem}\label{aplicaciones}

The subset sum problem is stated as follows. Let $D\subset \mathbb{F}_q$, $m, k\in\mathbb{N}$  such that $1\leq k\leq |D|$. Let $\boldsymbol{b} :=(b_1,\ldots,b_m)\in\mathbb{F}_q^m$, $N_m(k,\bfs{b},D)$ is the number of subsets $S\subset D$ with cardinal $k$ such that for $i=1,\ldots,m$, $\sum_{a\in S}a^i=b_i$.  The problem is to determine whether $N_m(k,\bfs{b}, D)$ is  positive. This problem is known as the {\textit{m}}--th moment {\textit{k}}--SSP.

However, when $D=\fq$ we can state it as a problem involving diagonal equations. More precisely, it is related with the problem of estimating the number of $\fq$--rational solutions of the following system:

\begin{equation*}\label{eq subset sum problem}
\left\{ \begin{array}{lcc}
             a_1^{1}+a_2^{1}+\cdots + a_k^{1}=b_1 \\
            \hspace{50px}\vdots\\ a_1^{m}+a_2^{m}+\cdots + a_k^{m}=b_m \\
             \end{array}
   \right.    
\end{equation*}

Observe that, two different solutions could represent the same subset (as one is a permutation of the other). On the other hand, if a solution verifies $a_i=a_j$ with $i\neq j$ that solution does not represent a subset.

In what follows, 
we shall consider a more general version of  the {\textit{m}}--th moment {\textit{k}}--SSP that we now state. 
Let $2\leq d_1<d_2<\ldots <d_m$, $D=\fq$, $m,k\in \N$ and $\boldsymbol{b}\in \fq^m$. Let $N_m(k,\boldsymbol{b}):=N_m(k,\boldsymbol{b},\fq)$ be the number of subsets $S$ of $\fq$ with $k$ elements such that $\sum_{a\in S}a^{d_j}=b_j$, for $1\leq j \leq m$.
Our main objective is to determine conditions over $q$, $k$ and $m$ under which $N_m(k,\boldsymbol{b})>0$.



In order to do this, we need to estimate the number of $\fq$--rational solutions of \eqref{eq: system}  with distinct coordinates. 

We first  recall some notations and definitions of permutations.
Let $S_k$  be the symmetric group of permutations of  $k$ elements. For a given permutation $\tau \in S_k$ we say that $\tau$ is of type  $(c_1,c_2,\ldots,c_k)$ if it has exactly $c_i$ cycles of length $i$, $1\leq i \leq k$. Note that $\sum_{i=1}^k ic_i=k$.
 Each $\tau\in S_k$ factors uniquely in product of disjoint cycles.
    We denote by  $l(\tau)$ the number of cycles in $\tau$, including the trivial cycles. Then $\mathrm{sign}(\tau)=(-1)^{k-l(\tau)}.$ Recall that given $\tau, \tau^{'} \in S_k$ are conjugated if they have the same type of cycle structure (up to the order). Let $C_k$ be the set of conjugacy classes of $S_k$. Let $C(\tau)$ be the number of conjugate permutations to $\tau$. If $\tau$ is of type $(c_1,c_2, \ldots,c_k)$ then $C(\tau)=\frac{k!}{1^{c_1}c_1! 2^{c_2} c_2! \cdots k^{c_k}c_k!}.$
    
    Let denote by $X$ the set of $\fq$--rational solutions of \eqref{eq: system} and let $\overline{X}\subset X$ be the set those solutions  with distinct coordinates.
    Given $\tau=(i_1i_2\cdots i_{\alpha_1})\cdots (l_1l_2\cdots l_{\alpha_s})$ we will define $X_\tau$ as follows $$X_\tau:=\{(x_1,\ldots,x_k)\in X,x_{i_1}=\cdots=x_{i_{\alpha_1}},\ldots,x_{l_1}=\cdots=x_{l_{\alpha_s}}\}.$$
 Each element of $X_{\tau}$ is said to be of type $\tau$.   
 Observe that $X$ is symmetric, namely $X$ is invariant under the action of $S_k$. 

We define the following generating function
\begin{align*}C_k(t_1, \ldots,t_k)&=\sum_{c_1+2c_2+\cdots+ kc_k=k} \frac{k!}{c_1! c_2! \cdots c_k!} \Big(\frac{t_1}{1}\Big)^{c_1}\Big(\frac{t_2}{2}\Big)^{c_2} \cdots \Big(\frac{t_k}{k}\Big)^{c_k}.
\end{align*}


 We shall prove some  properties of $C_k(t_1, \ldots,t_k)$.
 
\begin{lemma}
Let $C_k(t_1, \ldots,t_k)$ be the generating function defined above. The following identities hold.
\begin{itemize}
    \item If $t_1=q, \, \, t_2=-q,\, \cdots, t_k=(-1)^{k-1}q$, then
\begin{equation} \label{cuenta 1}
C_k(q,-q,\ldots,(-1)^{k-1}q)=\binom{q}{k} \cdot k!.
\end{equation}
\item If $t_i=\sqrt{q}$ for $1 \leq i \leq k$, then
\begin{equation}\label{cuenta 3}
C_k(\sqrt{q}, \cdots,\sqrt{q})=(-1)^k \binom{-\sqrt{q}}{k} k!.
\end{equation}
\end{itemize}
\end{lemma}
\begin{proof}
We consider the following exponential generating function \cite[proof of Lemma 4.1]{Wan2010}
\begin{equation}\label{serie generatriz}\sum_{k \geq 0} C_k(t_1,\ldots,t_k) \frac{u^k}{k!}=e^{ut_1+u^2\cdot \frac{t_2}{2}+u^3\cdot \frac{t_3}{3}+\cdots}.
\end{equation}


Thus, if $t_1=q, \, \, t_2=-q,\, \cdots, t_k=(-1)^{k-1}q$
\begin{align*}
C_k(q,-q,\ldots,(-1)^{k-1}q)&=  \sum_{\sum ic_i=k} \frac{k!}{c_1! c_2! \cdots c_k!} \Big(\frac{q}{1}\Big)^{c_1}\Big(\frac{-q}{2}\Big)^{c_2} \cdots \Big(\frac{(-1)^{k-1}q}{k}\Big)^{c_k}.
\end{align*}

On the other hand, from \eqref{serie generatriz}, we have that if $t_1=q, \, \, t_2=-q,\, \cdots, t_k=(-1)^{k-1}q$
\begin{align*}
\sum_{k \geq 0} C_k(q,\ldots,(-1)^{k-1}q) \frac{u^k}{k!}&=e^{q\cdot( u-\frac{u^2}{2}+\frac{u^3}{3}+\cdots)}\\
&=e^{q \ln(1+u)}=(1+u)^q\\
&= \sum_{k \geq 0} \frac{q!}{(q-k)!} \frac{u^k}{k!}.
\end{align*}
Thus,  the k--th term $C_k(q,-q,\ldots,(-1)^{k-1}q)$ of the serie  is $\frac{q!}{(q-k)!}$. Therefore,
 we obtained the identity \eqref{cuenta 1}.
 

In the same way, we have that
\begin{align*}
C_k(\sqrt{q},\sqrt{q},\ldots,\sqrt{q})&=  \sum_{\sum ic_i=k} \frac{k!}{c_1! c_2! \cdots c_k!} \Big(\frac{\sqrt{q}}{1}\Big)^{c_1}\Big(\frac{\sqrt{q}}{2}\Big)^{c_2} \cdots \Big(\frac{\sqrt{q}}{k}\Big)^{c_k} 
\end{align*}
Since \begin{align*}
\sum_{k \geq 0} C_k(\sqrt{q},\ldots,\sqrt{q}) \frac{u^k}{k!}&=e^{\sqrt{q}\cdot( u+\frac{u^2}{2}+\frac{u^3}{3}+\cdots)}\\
&=e^{\sqrt{q} (-\ln(1-u))}=(1-u)^{-\sqrt{q}}\\
&= \sum_{k \geq 0} (-1)^k \binom{-\sqrt{q}}{k} k! \frac{u^k}{k!},
\end{align*}
we deduce that $C_k(\sqrt{q},\ldots,\sqrt{q})=(-1)^k \binom{-\sqrt{q}}{k} k!.$
\end{proof}

\begin{theorem}\label{teo: estimación Nm(k,b) p no divide a k}
Suppose that $p$ does not divide $k$ and $p \geq 5$.  If $2 \leq m \leq \frac{k}{20}$, $k \leq 2q^{0.9}-\sqrt{q}+1$ and $q>2^{20}$, then
\[
\left |N_m(k,\boldsymbol{b})-\frac{1}{q^m}{q\choose k}\right|\leq M \cdot (-1)^k \binom{-\sqrt{q}}{k},
\]
where $M=3^4\cdot 2^{m-1}(3+d_mm)^{k+1}.$
\end{theorem}

\begin{proof}


Observe that $N_m(k,\boldsymbol{b})=\frac{|\overline{X}|}{k!}$ so, our first goal is to estimate the cardinal of $\overline{X}$.
Since $X$ is symmetric, from \cite[Proposition 2.5]{Wan2010}
we have that

\begin{equation*}\label{formula: X barra}
|\overline{X}|=\sum_{\tau\in C_k}(-1)^{k-l(\tau)}C(\tau)|X_\tau|,
\end{equation*}

where $C(\tau)=\frac{k!}{1^{c_1}c_1!\cdots k^{c_k}c_k!}.$ 

We first assume that $\boldsymbol{b}\neq 0$.
  \begin{align*}
     |\overline{X}|=\sum_{\tau\in C_k}(-1)^{k-l(\tau)}C(\tau)|X_\tau|&={\sum_{\tau\in CP_k}(-1)^{k-l(\tau)}C(\tau)|X_\tau|}+\sum_{\tau\notin CP_k}(-1)^{k-l(\tau)}C(\tau)|X_\tau|.
 \end{align*}
where $CP_k=\{\tau\in C_k:\text{the length of each cycle of $\tau$ is divisible by $p$}\}$. Then we can write  the complement of $CP_k$ in the following way:
$$CP_k^c=\displaystyle\bigcup_{i=0}^{\lfloor \frac{k}{p}\rfloor} \bigcup_{j=i+1}^kD_k(i,j),$$

where $D_k(i,j):=\{\tau \in C_k: \text{$\tau$ has $i$ cycles}\,  \text{with length  divisible by $p$ and}\, l(\tau)=j\}.$

Thus, 
\begin{align*}
|\overline{X}|&={\sum_{\tau\in CP_k}(-1)^{k-l(\tau)}C(\tau)|X_\tau|}+ \sum_{i=0}^{\lfloor \frac{k}{p}\rfloor} \, \, \sum_{\underset{\tau\in D_k(i,j)}{i+1 \leq j \leq k}} (-1)^{k-j} C(\tau)|X_\tau|.
 \end{align*}
Since $\bfs{b} \neq 0$ we deduce that    $|X_\tau|=0$ ($\forall \tau \in CP_k$). Therefore,
\begin{equation}\label{Suma a acotar}
|\overline{X}|= \sum_{i=0}^{\lfloor \frac{k}{p}\rfloor} \, \, \sum_{\underset{\tau\in D_k(i,j)}{i+1 \leq j \leq k}}  (-1)^{k-j}C(\tau)|X_\tau|.
 \end{equation}

Fix $i$ such that $ 0 \leq i \leq  \lfloor \frac{k}{p} \rfloor$. Let $\tau\in D_k(i, j)$. Thus, $X_{\tau}$ is the set of the $\fq$--rational solutions  of the variety $V_{\tau} \subset \A^{j-i}$ defined by a system of  $m$ diagonal equations with $j-i$ unknowns, which is similar to \eqref{eq: system}. We observe that:

Let $j$ such that $i+1 \leq j \leq 2m+i$. Since $m \geq \frac{j-i}{2}$, we  consider the first $\lfloor\frac{j-i-1}{2}\rfloor$ equations of the system and let $W\subset \A^{j-i}$ be the $\fq$--affine variety that these equations defined. Hence, with similar arguments of the proof of  Theorem \ref{teo: V es interseccion completa} and \eqref{eq: upper bound -- affine gral} we conclude that
\begin{equation}\label{Caso B}-d_m^m \cdot q^{\frac{j-i+1}{2}}\leq |X_\tau|\leq |W(\fq)|\leq d_m^m \cdot q^{\frac{j-i+1}{2}}.
\end{equation}

On the other hand, let $2m+i+1 \leq j \leq k$. With similar arguments as those developed in the proofs of Theorems \ref{teo: props de la clausura en el infinito} and \ref{teo: pcl V interseccion completa}  we have that $V_{\tau}$ has the geometric properties which can allow us  to apply  Theorem \ref{teo: estimacion ptos racionales V} in order to estimate $|X_\tau|$. We obtain that
\begin{equation}\label{Caso A}
\big||X_\tau|-q^{j-m}\big|\leq M_1 \cdot  q^{\frac{j}{2}},
\end{equation}
where $M_1=3^3\cdot 2^{m-1}(3+d_mm)^{k+1}.$

We first provide an upper bound for the number $|\overline{X}|$. From \eqref{Suma a acotar}, \eqref{Caso B} and \eqref{Caso A} we have that
\begin{align*} 
|\overline{X}|&= \sum_{i=0}^{\lfloor \frac{k}{p}\rfloor} \, \, \sum_{\underset{\tau\in D_k(i,j)}{i+1 \leq j \leq 2m+i}} (-1)^{k-j}C(\tau)|X_\tau|+\sum_{i=0}^{\lfloor \frac{k}{p}\rfloor} \, \, \sum_{\underset{\tau\in D_k(i,j)}{2m+i+1 \leq j \leq k}} (-1)^{k-j}C(\tau)|X_\tau|\\ \notag
& \leq \sum_{i=0}^{\lfloor \frac{k}{p}\rfloor} \, \, \sum_{\underset{\tau\in D_k(i,j)}{i+1 \leq j \leq 2m+i}} C(\tau)d_m^m q^{\frac{j-i+1}{2}}+\sum_{i=0}^{\lfloor \frac{k}{p}\rfloor} \, \, \sum_{\underset{\tau\in D_k(i,j)}{2m+i+1 \leq j \leq k}} (-1)^{k-j} C(\tau)q^{j-m}+ \\&+M_1 \cdot \sum_{i=0}^{\lfloor \frac{k}{p}\rfloor} \, \, \sum_{\underset{\tau\in D_k(i,j)}{2m+i+1 \leq j \leq k}} C(\tau) q^{j/2}\\ \notag
 &\leq \sum_{\underset{\tau\in D_k(0,j)}{1\leq j \leq 2m}} C(\tau)d_m^m q^{\frac{j+1}{2}}+ \sum_{i=0}^{\lfloor \frac{k}{p}\rfloor} \, \, \sum_{\underset{\tau\in D_k(i,j)}{i+1 \leq j \leq 2m+i}} C(\tau)d_m^m q^{\frac{j}{2}}+\\ \notag
 &+\sum_{i=0}^{\lfloor \frac{k}{p}\rfloor} \, \sum_{\underset{\tau\in D_k(i,j)}{2m+i+1 \leq j \leq k}} (-1)^{k-j} C(\tau)q^{j-m} + M_1 \cdot \sum_{i=0}^{\lfloor \frac{k}{p}\rfloor} \, \, \sum_{\underset{\tau\in D_k(i,j)}{2m+i+1 \leq j \leq k}} C(\tau) q^{j/2} \\ \notag
  &\leq \sum_{\underset{\tau\in D_k(0,j)}{1\leq j \leq 2m}} C(\tau)d_m^m q^{\frac{j+1}{2}}+ M_1 \cdot \sum_{i=0}^{\lfloor \frac{k}{p}\rfloor} \, \, \sum_{\underset{\tau\in D_k(i,j)}{i+1 \leq j \leq k}} C(\tau) q^{\frac{j}{2}}+\sum_{i=0}^{\lfloor \frac{k}{p}\rfloor} \, \sum_{\underset{\tau\in D_k(i,j)}{2m+i+1 \leq j \leq k}} (-1)^{k-j} C(\tau)q^{j-m} \\ \notag
&\leq d_m^m \sum_{0\leq j \leq 2m} \binom{j+k-1}{j} q^{\frac{j+1}{2}}+ M_1 \cdot \sum_{i=0}^{\lfloor \frac{k}{p}\rfloor} \, \, \sum_{\underset{\tau\in D_k(i,j)}{i+1 \leq j \leq k}} C(\tau) q^{\frac{j}{2}}+\\ \notag 
&+\sum_{i=0}^{\lfloor \frac{k}{p}\rfloor} \, \sum_{\underset{\tau\in D_k(i,j)}{2m+i+1 \leq j \leq k}} (-1)^{k-j} C(\tau)q^{j-m} \\ \notag 
\end{align*}
Taking into account the following identity (see, for example, \cite[Chapter 5, Section 1]{Kn94})
\begin{equation} \label{combinatorio knuth}\sum_{k=0}^n \binom{m+k}{k}=\binom{m+n+1}{n},\end{equation}
we have that 
\begin{equation} \label{combinatorio}\sum_{0\leq j \leq 2m} \binom{j+k-1}{j} = \binom{k+2m}{2m} \leq 2^{k+2m}. \end{equation}

From \eqref{combinatorio} and \eqref{cuenta 3} we obtain
\begin{align} \label{cuenta suma 1}
|\overline{X}|
& \leq d_m^m \, \,   2^{k+2m} q^{\frac{2m+1}{2}}+\sum_{i=0}^{\lfloor \frac{k}{p}\rfloor} \, \, \sum_{\underset{\tau\in D_k(i,j)}{2m+i+1 \leq j \leq k}}\! \! \! \! (-1)^{k-j} C(\tau)q^{j-m}+ M_1 \cdot (-1)^k \binom{-\sqrt{q}}{k} k! \\ \notag
&  \leq  \frac{1}{q^m}\Big(\sum_{ \underset{ l(\tau)=c_1+ \ldots+c_k}{\sum i \dot c_i=k}} (-1)^{k-l(\tau)}C(\tau) q^{l(\tau)}\!\!- \sum_{i=0}^{\lfloor \frac{k}{p}\rfloor} \, \, \sum_{\underset{\tau\in D_k(i,j)}{i+1 \leq j \leq 2m+i}} (-1)^{k-j} C(\tau)q^{j} \\ \notag
&-\sum_{j=1}^k (-1)^{k-j} p(k,j) q^j\Big)+d_m^m   2^{k+2m} \, q^{\frac{2m+1}{2}}+ M_1 \cdot (-1)^k \binom{-\sqrt{q}}{k} k!,
\end{align}
where $p(k,j)$ is the number of permutations in $S_k$ with $j$ cycles and the length of all of them is divisible by $p$.

Since $p$ does not divide $k$, from \cite[Lemma 3.1]{Nu2019}, we have that 
$$\sum_{j=1}^k (-1)^{k-j} p(k,j) q^j=0.$$
Thus, from  \eqref{cuenta suma 1}, \eqref{cuenta 1} and \eqref{combinatorio knuth} we deduce that
\begin{align*}
|\overline{X}| &\leq \frac{1}{q^m}\binom{q} {k} k! - \sum_{i=0}^{\lfloor \frac{k}{p}\rfloor} \, \, \sum_{\underset{\tau\in D_k(i,j)}{i+1 \leq j \leq 2m+i}} (-1)^{k-j} C(\tau)q^{j-m} +  M_1 \cdot (-1)^k \binom{-\sqrt{q}}{k} k!+\\ \notag
&+d_m^m \, 2^{k+2m} \, q^{\frac{2m+1}{2}} \\ \notag
&\leq \frac{1}{q^m}\binom{q} {k} k! + q^{m+ \lfloor \frac{k}{p}\rfloor} \sum_{i=0}^{\lfloor \frac{k}{p}\rfloor} \binom{k+2m+i}{2m+i}+  M_1 \cdot (-1)^k \binom{-\sqrt{q}}{k} k!+d_m^m \, 2^{k+2m} \, q^{\frac{2m+1}{2}} \\ \notag
&\leq \frac{1}{q^m}\binom{q} {k} k! + q^{m+ \lfloor \frac{k}{p}\rfloor}  \binom{k+2m+\lfloor \frac{k}{p}\rfloor+1}{2m+\lfloor \frac{k}{p}\rfloor}+  M_1 \cdot (-1)^k \binom{-\sqrt{q}}{k} k!+d_m^m \, 2^{k+2m} \, q^{\frac{2m+1}{2}} \\ \notag
&\leq \frac{1}{q^m}\binom{q} {k} k! + q^{m+ \lfloor \frac{k}{p}\rfloor}\, 2^{k+2m+\lfloor \frac{k}{p}\rfloor+1}+  M_1 \cdot (-1)^k \binom{-\sqrt{q}}{k} k!+d_m^m \, 2^{k+2m} \, q^{\frac{2m+1}{2}}
\end{align*}
From the Remarks \ref{Observacion 1}, \ref{Observacion 2}, and under the conditions over $p, q, k$ and $m$ described in the theorem, the following upper bound for $|\overline{X}|$ is deduced:
\begin{equation} \label{cota superior de X barra con b distinto de cero}
|\overline{X}| \leq \frac{1}{q^m}\binom{q} {k}\cdot k! +  M \cdot (-1)^k \binom{-\sqrt{q}}{k} k!,
\end{equation}
where $M=3^4\cdot 2^{m-1}(3+d_mm)^{k+1}.$

Now we provide a lower bound for the number $|\overline{X}|$.
From \eqref{cuenta 3}, \eqref{Suma a acotar}, \eqref{Caso B} and \eqref{Caso A} we have that
\begin{align*} 
|\overline{X}|&= \sum_{i=0}^{\lfloor \frac{k}{p}\rfloor} \, \, \sum_{\underset{\tau\in D_k(i,j)}{i+1 \leq j \leq 2m+i}} (-1)^{k-j} 
C(\tau)|X_\tau|+\sum_{i=0}^{\lfloor \frac{k}{p}\rfloor} \, \, \sum_{\underset{\tau\in D_k(i,j)}{2m+i+1 \leq j \leq k}} (-1)^{k-j} C(\tau)|X_\tau|\\ \notag
& \geq - \sum_{i=0}^{\lfloor \frac{k}{p}\rfloor} \, \, \sum_{\underset{\tau\in D_k(i,j)}{i+1 \leq j \leq 2m+i}} C(\tau)d_m^m q^{\frac{j-i+1}{2}}+\sum_{i=0}^{\lfloor \frac{k}{p}\rfloor} \, \, \sum_{\underset{\tau\in D_k(i,j)}{2m+i+1 \leq j \leq k}} (-1)^{k-j} C(\tau)q^{j-m}\\ \notag
&-M_1 \cdot \sum_{i=0}^{\lfloor \frac{k}{p}\rfloor} \, \, \sum_{\underset{\tau\in D_k(i,j)}{2m+i+1 \leq j \leq k}} C(\tau) q^{j/2} \\ \notag
&\geq  -d_m^m \sum_{0\leq j \leq 2m} \binom{j+k-1}{j} q^{\frac{j+1}{2}}- M_1 \cdot \sum_{i=0}^{\lfloor \frac{k}{p}\rfloor} \, \, \sum_{\underset{\tau\in D_k(i,j)}{i+1 \leq j \leq k}} C(\tau) q^{\frac{j}{2}}+\\ \notag 
&+ \sum_{i=0}^{\lfloor \frac{k}{p}\rfloor} \, \, \sum_{\underset{\tau\in D_k(i,j)}{2m+i+1 \leq j \leq k}} (-1)^{k-j} C(\tau)q^{j-m}. 
\end{align*}
By similar arguments to those presented above, we deduce that
\begin{align} \label{cota inferior de X barra con b distinto de cero}
|\overline{X}|& \geq \frac{1}{q^m}\binom{q} {k}\cdot k! -  M_1 \cdot (-1)^k \binom{-\sqrt{q}}{k} k!- q^{m+ \lfloor \frac{k}{p}\rfloor}\, 2^{k+2m+\lfloor \frac{k}{p}\rfloor+1}-d_m^m \, 2^{k+2m} \, q^{\frac{2m+1}{2}}\\ \notag
& \geq \frac{1}{q^m}\binom{q} {k}\cdot k! -  M \cdot (-1)^k \binom{-\sqrt{q}}{k} k!.
\end{align}

Then from 
 \eqref{cota superior de X barra con b distinto de cero} and \eqref{cota inferior de X barra con b distinto de cero} we obtain the estimate in the case $\boldsymbol{b}\neq \boldsymbol{0}$.
 
Let's now suppose that $\bfs{b}=0$. We recall that
 \begin{align*}
     |\overline{X}|=\sum_{\tau\in C_k}(-1)^{k-l(\tau)}C(\tau)|X_\tau|&={\sum_{\tau\in CP_k}(-1)^{k-l(\tau)}C(\tau)|X_\tau|}+\sum_{\tau\notin CP_k}(-1)^{k-l(\tau)}C(\tau)|X_\tau|,
 \end{align*}
 where $CP_k=\{\tau\in C_k: \text{  the length of each cycle of $\tau$ is divisible by $p$}\}$.
 
Since $\bfs{b}=0$, we have that $|X_\tau|=q^{l(\tau)}$  ($\forall \tau \in CP_k$). Thus,
 \begin{equation*}
|\overline{X}|=\sum_{\tau\in C_k}(-1)^{k-l(\tau)}C(\tau)|X_\tau|=\sum_{\tau\in CP_k}(-1)^{k-l(\tau)}C(\tau)q^{l(\tau)}+\sum_{\tau\notin CP_k}(-1)^{k-l(\tau)}C(\tau)|X_\tau|.
 \end{equation*}
We observe that
\begin{equation*}\sum_{\tau\in CP_k}(-1)^{k-l(\tau)}C(\tau)q^{l(\tau)}=\sum_{j=1}^k (-1)^{k-j} p(k,j) q^j=0.
\end{equation*}
Therefore, 
$$|\overline{X}|=\sum_{\tau\notin CP_k}(-1)^{k-l(\tau)}C(\tau)|X_\tau|.$$
By similar calculations as above, we deduce the bounds \eqref{cota superior de X barra con b distinto de cero} and \eqref{cota inferior de X barra con b distinto de cero} in this case.
Finally, from \eqref{cota superior de X barra con b distinto de cero} and \eqref{cota inferior de X barra con b distinto de cero}, we conclude the theorem's result.
 
\end{proof}

\begin{theorem}\label{teo: estimación Nm(k,b) p  divide a k}
Suppose that $p$ divides $k$ and $p \geq 3$.  If $ m \leq \frac{k}{20}$, $k \leq 2q^{0.9}-\sqrt{q}+1$ and $q \geq 2^{21}$, then
\[
\left |N_m(k,\boldsymbol{b})-\frac{1}{q^m}\Big({q\choose k} +(-1)^{k+k/p} v(\bfs{b}) \binom{q/p}{k/p}\Big)\right|\leq M \cdot (-1)^k \binom{-\sqrt{q}}{k},
\]
where $v(b)=q^m-1$  if $\bfs{b}=0$, $v(\bfs{b})=-1$ if $\bfs{b} \neq 0$ and $M=3^4\cdot 2^{m-1}(3+d_mm)^{k+1}.$
\end{theorem}
\begin{proof}

We first consider the case $\bfs{b} \neq 0$. Following the ideas of the proof of the theorem above, we have that
\begin{align*}
|\overline{X}|=\sum_{\tau\in C_k}(-1)^{k-l(\tau)}C(\tau)|X_\tau|&={\sum_{\tau\in CP_k}(-1)^{k-l(\tau)}C(\tau)|X_\tau|}+ \sum_{i=0}^{ \frac{k}{p}} \, \, \sum_{\underset{\tau\in D_k(i,j)}{i+1 \leq j \leq k}} (-1)^{k-j} C(\tau)|X_\tau|,
 \end{align*}
 where $CP_k=\{\tau\in C_k: \text{  the length of each cycle of $\tau$ is divisible by $p$}\}$ and $D_k(i,j):=\{\tau \in C_k: \text{$\tau$ has $i$ cycles}\,  \text{with length  divisible by $p$ and}\, l(\tau)=j\}.$

 Since $\bfs{b} \neq 0$ we deduce that    $|X_\tau|=0$ ($\forall \tau \in CP_k$). Therefore,
\begin{equation*}
|\overline{X}|= \sum_{i=0}^{ \frac{k}{p}} \, \, \sum_{\underset{\tau\in D_k(i,j)}{i+1 \leq j \leq k}} (-1)^{k-j} C(\tau)|X_\tau|.
 \end{equation*}
 Following the same steps as the proof of the theorem above, we have that
\begin{align*}|\overline{X}|& \leq \frac{1}{q^m}\Big(\sum_{ \underset{ l(\tau)=c_1+ \ldots+c_k}{\sum i \dot c_i=k}} (-1)^{k-l(\tau)}C(\tau) q^{l(\tau)}\!\!- \sum_{i=0}^{\frac{k}{p}} \, \, \sum_{\underset{\tau\in D_k(i,j)}{i+1 \leq j \leq 2m+i}} (-1)^{k-j} C(\tau)q^{j}\\ \notag
& -\sum_{j=1}^k (-1)^{k-j} p(k,j) q^j\Big)+d_m^m   2^{k+2m} \, q^{\frac{2m+1}{2}}+ M_1 \cdot (-1)^k \binom{-\sqrt{q}}{k} k!.
\end{align*}
Since $p$  divides $k$, from \cite[Lemma 3.1]{Wan2010}
$$\sum_{j=1}^k (-1)^{k-j} p(k,j) q^j=(-1)^{k+k/p} \binom{q/p}{k/p} k!.$$
Then, from the equality above and \eqref{cuenta 1}  we obtain
\begin{equation}\begin{split} \label{cota superior p divide a k b distinto de cero}
|\overline{X}|&\leq \frac{1}{q^m}\Big(\binom{q}{k} k!-(-1)^{k+k/p} \binom{q/p}{k/p} k!\Big)+q^{m+  \frac{k}{p}}\, 2^{k+2m+\frac{k}{p}+1} +d_m^m   2^{k+2m} \, q^{\frac{2m+1}{2}}+\\
&+ M_1 \cdot (-1)^k \binom{-\sqrt{q}}{k} k!.
\end{split}
\end{equation}
By similar reasoning to the proof of the theorem above, we deduce the following lower bound for $|\overline{X}|$.
\begin{equation}\begin{split} \label{cota inferior p divide a k b distinto de cero}
|\overline{X}|&\geq \frac{1}{q^m}\Big(\binom{q}{k} k!-(-1)^{k+k/p} \binom{q/p}{k/p} k!\Big)-q^{m+  \frac{k}{p}}\, 2^{k+2m+\frac{k}{p}+1}- d_m^m   2^{k+2m} \, q^{\frac{2m+1}{2}}\\
&- M_1 \cdot (-1)^k \binom{-\sqrt{q}}{k} k!.
\end{split}
\end{equation}
Now, we suppose that $\bfs{b}=0$. 
We recall that
 \begin{align*}
     |\overline{X}|=\sum_{\tau\in C_k}(-1)^{k-l(\tau)}C(\tau)|X_\tau|&={\sum_{\tau\in CP_k}(-1)^{k-l(\tau)}C(\tau)|X_\tau|}+\sum_{i=0}^{ \frac{k}{p}} \, \, \sum_{\underset{\tau\in D_k(i,j)}{i+1 \leq j \leq k}} (-1)^{k-j} C(\tau)|X_\tau|.
 \end{align*}
Since $\bfs{b}=0$, $|X_\tau|=q^{l(\tau)}$  ($\forall \tau \in CP_k$). Thus,
 \begin{equation*}
|\overline{X}|=\sum_{\tau\in C_k}(-1)^{k-l(\tau)}C(\tau)|X_\tau|=\sum_{\tau\in CP_k}(-1)^{k-l(\tau)}C(\tau)q^{l(\tau)}+\sum_{i=0}^{ \frac{k}{p}} \, \, \sum_{\underset{\tau\in D_k(i,j)}{i+1 \leq j \leq k}} (-1)^{k-j}C(\tau)|X_\tau|.
 \end{equation*}
Since $p$ divides $k$, then
\begin{align*}\sum_{\tau\in CP_k}(-1)^{k-l(\tau)}C(\tau)q^{l(\tau)}
&=\sum_{j=1}^k (-1)^{k-j} p(k,j) q^j=(-1)^{k+k/p} \binom{ q/p}{k/p} k!.
\end{align*}
Then, we obtain
 \begin{equation}\label{cuenta para la suma con p divide a k}
|\overline{X}|=\sum_{\tau\in C_k}(-1)^{k-l(\tau)}C(\tau)|X_\tau|=(-1)^{k+k/p} \binom{ q/p}{k/p} k!+\sum_{\tau\notin CP_k}(-1)^{k-l(\tau)}C(\tau)|X_\tau|.
 \end{equation}
 The second term of the right hand of  \eqref{cuenta para la suma con p divide a k} is bounded in the same way as in the Theorem \ref{teo: estimación Nm(k,b) p no divide a k}.
 Finally, from \eqref{cota superior p divide a k b distinto de cero}, \eqref{cota inferior p divide a k b distinto de cero}, \eqref{cuenta para la suma con p divide a k}
 and, taking into account  Remarks  \ref{remark 7.4} and \ref{Observacion 2}, 
 we deduce the statement of the theorem.

\begin{remark}
    
Let's note that under our hypothesis, our results are an improvement with regard to \cite{Wan2010}. Indeed, our error is of the order $\mathcal{O}(p^{\frac{sk}{2}})$ while in \cite{Wan2010} is of the order $\mathcal{O}(p^{(s-1)k})$. Also, in the case that $p$ divides $k$, we have an aditional term in the theorical estimate.

Moreover we can apply our results to a more general version of the problem, as we are not constrained to consecutive exponents.
\end{remark}

\begin{remark}
Let $p \geq 3$, $m \leq \frac{k}{30}+ \frac{2}{3}$ and $k \leq 2 q^{0.9}-\sqrt{q}+1$. With similar arguments to the proofs of Theorems \ref{teo: estimación Nm(k,b) p no divide a k} and \ref{teo: estimación Nm(k,b) p  divide a k}, and using  \eqref{systems homogeneous}, we deduce that

\begin{enumerate}
\item If $p$ does not divide $k$, then
\[
\left |N_m(k,\boldsymbol{0})-\frac{1}{q^m}{q\choose k}\right|\leq (\frac{7}{2}m^2)^{k+1} \cdot (-1)^k \frac{1}{q^{\frac{m-1}{2}}}\binom{-\sqrt{q}}{k}.
\]
\item If $p$ divides $k$, then
\[
\left |N_m(k,\boldsymbol{0})-\frac{1}{q^m}\Big({q\choose k} +(-1)^{k+k/p} q^{m-1} \binom{q/p}{k/p}\Big)\right|\leq (\frac{7}{2}m^2)^{k+1} \cdot (-1)^k \frac{1}{q^{\frac{m-1}{2}}}\binom{-\sqrt{q}}{k}.
\]
\end{enumerate}

\end{remark}

\end{proof}
\subsection{$m$--th moment $k$-SSP for large $k$ and medium $m$}
 Now we aim to analyse when $N_m(k,\bfs{b})>0$ occurs. We first consider the cases for large values of $k$.



By analyzing our bound we can state the following complementing result.

\begin{theorem} \label{teo: existencia con estimación} 
Let $p \geq 3$ such that $p$ does not divide $k$ or $p$ divides $k$  and $\bfs{b}=\bfs{0}.$

If $ m \leq d_m \leq \frac{\sqrt{3}}{3} k^{0.02}$, $k \leq 2q^{0.9}-\sqrt{q}+1$ and $q>2^{20}$,
then $N_m(k,\bfs{b})>0$ for all $\bfs{b}\in\fq^m$.

\end{theorem}

\begin{proof}

We observe that, if $m \leq k-5$ and $k \geq 7$, we obtain
$$3^4\cdot 2^{m-1}(3+d_mm)^{k+1} \leq (3 d_m m)^{k+1}.$$
Thus, we have that
\begin{align*}
    N_m(k,\bfs{b})>0\Leftrightarrow {q\choose k}\frac{1}{q^m}-\left(3md_m\right)^{k+1}(-1)^k{-\sqrt{q}\choose k}>0
\end{align*}
From Remark \ref{Observacion 1} and taking into account that $k \leq 2q^{0.9}-\sqrt{q}+1$, $q>2^{20}$ and $2 \leq m \leq \frac{k}{20}$
we have that
 $\frac{{q\choose k}}{(-1)^k{-\sqrt{q}\choose k}}\geq \frac{q^{0.1 k}}{2^k} \geq q^{\frac{1}
 {20}k}$. Then,  we only need to prove  that $q^{\frac{1}{20}k-m}-\left(3md_m\right)^{k+1}>0$, which is obtained straightforward from the conditions over 
 $k$, $q$ and $m$.

\end{proof}

\begin{theorem} \label{teo: existencia con estimación 2} 

Let $p \geq 3$ such that $p$ divides $k$ and $\bfs{b} \neq 0$.
If $ m \leq d_m \leq \frac{\sqrt{3}}{3} k^{0.02}$, $k \leq 2q^{0.9}-\sqrt{q}+1$ and $q \geq 2^{21}$,
then $N_m(k,\bfs{b})>0$ for all $\bfs{b}\in\fq^m$.
\end{theorem}
\begin{proof}
We have that
\begin{align*}
    N_m(k,\bfs{b})>0\Leftrightarrow \Big({q\choose k}- \binom{q/p}{k/p}\Big)\frac{1}{q^m}-\left(3 md_m\right)^{k+1}(-1)^k{-\sqrt{q}\choose k}>0.
\end{align*}
From Remark \ref{Observacion 1} and taking into account that  $k \leq 2q^{0.9}-\sqrt{q}+1$ and $q \geq 2^{21}$ we have that 
 $\frac{{q\choose k}}{(-1)^k{-\sqrt{q}\choose k}}\geq \frac{q^{0.1 k}}{2^k}$.
We first consider the case $k>p$. From Remark \ref{observacion para el caso p divide a k} we obtain that $\binom{q/p}{k/p}\leq \binom{q}{k}\frac{1}{q^{(p-1)/10}}$.
     Then, taking into account that $p \geq 3$, $ k\leq 2q^{0.9}$, $q \geq 2^{21}$ we deduce that $${q\choose k}- \binom{q/p}{k/p}\geq \binom{q}{k}\Big(1-\frac{1}{q^{(p-1)/10}}\Big)\geq \frac{q^{0.1 k}}{2^{k+1}}(-1)^k{-\sqrt{q}\choose k}\geq q^{\frac{11}{210}k-\frac{1}{21}}(-1)^k{-\sqrt{q}\choose k}.$$
 
Then,  we only need to check that $q^{\frac{11}{210}k-\frac{1}{21}-m}-\left(3 md_m\right)^{k+1}>0$.
But this inequality holds from the conditions over $k$, $q$ and $m$ which  we have assumed, the last inequality holds.
Finally with similar arguments above, we can prove  the case $k=p$ using that $\binom{q}{k}-\frac{q}{k} \geq \binom{q}{k}(1-\frac{2}{q^{0.1}}) \geq \binom{q}{k} \frac{1}{2}.$

\end{proof}
In \cite{Wan2010} the authors present the following result:

\begin{theorem} \label{existencia Wan}
For any $\varepsilon>0$, there is a constant $c_{\varepsilon}>0$ such that if $m<\varepsilon \sqrt{k}$ and $4\varepsilon^2\ln^2{q}<k<c_{\varepsilon
}q$, then $N_m(k,\bfs{b})>0$ for all $\bfs{b}\in \fq^m$.
\end{theorem}
\begin{remark} Our results complements Theorem \ref{existencia Wan}. 
Indeed, from the hypothesis in \ref{existencia Wan}, is clear that $\varepsilon > \frac{m}{\sqrt{k}}$ and $4\varepsilon^2\ln^2q < k$. Consequently, $q<e^{\frac{k}{2m}}.$
Thus if $k$ and $m$ are fixed, there exists an upper bound on $q$. However, with our results given a specific $k$ and $m$, one can choose an arbitrarily large value for $q$.

On the other hand one can deduce from the hypothesis that $2m \ln{q}<k$
also indicates that, for a fixed $q$, $k$ must be sufficiently large. In our theorems, $q$ does not impose a lower bound on $k$.

Finally, we can apply our results to a more general version of the problem, as we are not constrained to consecutive exponents.

\end{remark}


\subsection{$m$--th moment $k$-SSP for medium $k$ and large $m$}
In this section we present an existence's result based on the Brun sieve in order to obtain a low bound of the number of $\fq$-rational solutions of the system with at least two equals coordinates.
The Brun sieve tell us that
$$N_m(k,\bfs{b})\geq |V(\fq)|-\binom{k}{2}|V_{1,2}(\fq)|,$$
where $V_{1,2}= V\cap \{X_1=X_2\}$.
According to Theorem \ref{teo: estimacion ptos racionales V}, if $m\leq \frac{k-2}{2}$ we have that
$$\left ||V(\mathbb{F}_q)|-q^{k-m}\right|\leq 3^3\cdot 2^{m-2}(3+d_mm)^{k+1}\cdot q^{\frac{k}{2}},$$
and $$\left ||V_{1,2}(\mathbb{F}_q)|-q^{k-m-1}\right|\leq 3^3\cdot 2^{m-2}(3+d_mm)^{k}\cdot q^{\frac{k-1}{2}}.$$
On the other hand, it is easy to show that if $m\leq k-3$  then  $3^3\cdot 2^{m-2}(3+d_mm)^{k+1}\leq  (3d_mm)^{k+1}$ and $3^3\cdot 2^{m-2}(3+d_mm)^{k}\leq  (3d_mm)^{k}$.
Then for $m\leq \frac{k-2}{2}$  we obtain that
\begin{align*}
    N_m(k,\bfs{b})\geq & q^{k-m}- \Big(3d_mm\Big)^{k+1}q^{\frac{k}{2}}-\frac{k^2}{2}\Big(q^{k-m-1}+\Big(3d_mm\Big)^{k}q^{\frac{k-1}{2}}\Big)\\
\end{align*}
Therefore, $N_m(k,\bfs{b})>0$ if  \begin{equation}\label{ec: desigualdad 1 Brun}
    q^{\frac{k-1}{2}-m}\Big(q-\frac{k^2}{2}\Big)>\Big(3d_mm\Big)^{k}\Big(3d_mm q^{\frac{1}{2}}+\frac{k^2}{2}\Big).
\end{equation}
We observe that $\frac{k-25}{50}\leq \frac{\sqrt{3}}{3} k$ for all $k \geq 1$. Therefore, if $  m \leq d_m \leq \frac{k-25}{50}$  and $k \leq q^{0.24}$ then $3m d_m \leq k^2 \leq q^{0.48}.$ Thus, we have that if $q\geq 106,$
$$3d_m m q^{1/2} +k^ 2<q.$$
Thus, in order to prove that \eqref{ec: desigualdad 1 Brun} holds, it is sufficient to ask that 
\begin{equation}\label{ec: Brun desigualdad 2}
  q^{\frac{k-1}{2}-m}>\Big(3d_mm\Big)^{k}.
\end{equation}
Observe that, if $ m \leq d_m \leq\frac{k-25}{50}$ and $k \leq q^{0.24}$  then  \eqref{ec: Brun desigualdad 2} holds.

We have proved the following result.
\begin{theorem} \label{teo: existencia con Brun}
    Let $ m\leq d_m \leq \frac{k-25}{50}$ and $ k\leq q^{0,24}$
    then $N_m(k,\bfs{b})>0$ for all $\bfs{b}\in \fq^m$.
\end{theorem}



\begin{remark}
We compare Theorem \ref{teo: existencia con Brun} with Theorems \ref{teo: existencia con estimación} and \ref{teo: existencia con estimación 2}. We observe that both results are complementary. The goal of Theorem \ref{teo: existencia con Brun} is to establish a linear relationship between $m$ and $k$. Hence, this theorem allows us to consider a wider range of values for $m$  which ensures that $N_m(k,\bfs{b})>0$. On the other hand, Theorems \ref{teo: existencia con estimación} and \ref{teo: existencia con estimación 2} are better than Theorem \ref{teo: existencia con Brun} because $k \leq 2q^{0,9}-\sqrt{q}+1$ while in Theorem \ref{teo: existencia con Brun},  $k\leq q^{0,24}.$ 
\end{remark}

\section{Moment Subset Sums Problem with $D= \{f(x):\,\,x\in \fq\}$.}
In this section we show how our techniques can we adapted to study the $m$--th moment $k$--SSP when the evaluation set $D$ is the image set of a polynomial $f\in \fq[T]$.

Let $D\subset \mathbb{F}_q$, $m, k\in\mathbb{N}$  such that $1\leq k\leq |D|$. Let $\boldsymbol{b} :=(b_1,\ldots,b_m)\in\mathbb{F}_q^m$, recall that $N_m(k,\bfs{b},D)$ is the number of subsets $S\subset D$ with cardinal $k$ such that for $i=1,\ldots,m$, $\sum_{a\in S}a^i=b_i$.  The problem is to determine whether $N_m(k,\bfs{b}, D)$ is  positive.

Let $q=p^s$ with $p\neq 2$, $p$ does not divide $ 3\cdots (m-1)$ and $m\leq \frac{k-1}{2}$. We consider a polynomial $f\in \fq[T]$ such that $\deg(f)=n$, $f=a_nT^n+\cdots + a_2T^2$ and let $D=\{f(x):\,\,x\in \fq\}$ be the image set of $f$. We define the following set:
$$A:=\{(y_1,\ldots,y_k)\in D^k:\,\sum_{l=1}^k y_l^i=b_i\,\,, 1\leq i \leq m,\,y_h\neq y_l,\, h\neq l\},$$
then we have that $N_m(k,\bfs{b},D)=\frac{|A|}{k!}.$ On the other hand, we define the set $B$ as follows:
$$B:=\{(x_1,\ldots,x_k)\in \fq^k:\,\sum_{l=1}^k f(x_l)^i=b_i\,\,, 1\leq i \leq m,\,x_h\neq x_l,\, \text{and}\, f(x_h)\neq f(x_l),\,\,\, h\neq l\}.$$
We claim that $|A|=|B|$. Let $(y_1,\ldots,y_k)\in A$, there exist $x_i\in\fq$ such that $f(x_i)=y_i$, $1\leq i \leq k$ and, taking into account that $y_i\neq y_j$, we deduce that  $f(x_i)\neq f(x_j)$ and $x_i\neq x_j$. Hence, $(x_1,\ldots,x_k)\in B.$
On the other hand, for a given $(x_1,\ldots,x_k)\in B$, let $y_i=f(x_i)$, $1\leq i \leq k$ since $x_i\neq x_j$ and $f(x_i) \neq f(x_j)$ we conclude that $(y_1,\ldots,y_k)\in A$.
 Now, we shall estimate the number $|B|$. Let $\mathcal{S}$ be the set of $\fq$-rational solutions of the system:
\begin{equation}\label{ sistema con los f}
\left\{ \begin{array}{lcc}
             f(X_1)+f(X_2)+\cdots + f(X_k)=b_1 \\
            \hspace{50px}\vdots\\ f(X_1)^{m}+f(X_2)^{m}+\cdots + f(X_k)^{m}=b_m \\
             \end{array}
   \right.    
\end{equation}
Then we have that
$$B=\mathcal{S}\setminus \{(x_1,\ldots,x_k)\in \mathcal{S}:\,\,f(x_i)=f(x_j),\,\,1\leq i<j\leq k\}.$$
In order to estimate $|\mathcal{S}|$ we shall consider  the following polynomials 
 $g_j\in \fq[X_1, \ldots,X_k]$
 \begin{equation}\label{def: f'i}
 g_j:=f(X_1)^{j}+f(X_2)^{j}+\cdots+f(X_k)^{j}-b_j,\,\,1 \leq j \leq m.
 \end{equation}
Let $V_g:=V(g_1, \ldots,g_m)\subset \A^k$ be the $\fq$--affine variety defined by $g_1, \ldots,g_m$. We shall study some facts concerning the geometry of $V_g$. 
\begin{lemma} \label{lemma: codimension del lugar singula de Vg}
 Let $m\leq \frac{k-1}{2}$, the codimension of the singular locus of $V_g$ is at least two and the ideal $\left(g_1,\ldots,g_m\right)$ is radical.

\end{lemma}
\begin{proof} The proof is similar to the proof of Lemma \ref{lemma: codimension del lugar singula de V}.
 The critical point is to study the dimension of the following set:
\[
\left\{\boldsymbol{x}\in V_g:rank\left(\frac{\partial g}{\partial \boldsymbol{X}}\right)\Big |_{\boldsymbol{x}} <m\right\}=\bigcup_{i=0}^{m-1} C_i,
\]

where  $\left(\frac{\partial g}{\partial \boldsymbol{X}}\right)$ is the Jacobian $(m\times k)$--matrix of $g_1,\ldots,g_m$ with respect to $X_1,\ldots,X_k$, and     $C_i=\left\{\boldsymbol{x}\in V_g: rank\left(\frac{\partial g}{\partial \boldsymbol{X}}\right)\Big |_{\boldsymbol{x}}=i\right\}.$

Given an $\boldsymbol{x}\in C_i$, there is a minor of $\left(\frac{\partial g}{\partial \boldsymbol{X}}\right)|_{\boldsymbol{x}}$  of size $i\times i$ which is not zero. 
More precisely, without loss of generality, we can suppose that
 the $i\times i$--submatrix $B_i(\boldsymbol{X})$ which consists of the first $i$ rows and the first $i$
columns of $\left(\frac{\partial g}{\partial \boldsymbol{X}}\right)$, 
namely, $(B_i(\bfs{X}))_{l,j}=lf(X_j)^{l-1}f'(X_j)$, $1\leq j,l\leq i$,
and  $\det(B_i(\boldsymbol{x}))\neq 0$ holds.
Then $f'(x_l)\neq 0$ for $l=1,\ldots, i$ and $f(x_l)\neq f( x_j)$ if $l\neq j$.

On the other hand, we consider the $(i+1)\times(i+1)$--matrices $A_j(\boldsymbol{X})$ define by
\[A_j(\boldsymbol{X})=
\begin{pmatrix}
    &  & & \vline & f'(X_j)\\
     & B_i(\boldsymbol{X}) & & \vline & \vdots\\
    & & & \vline  & if(X_j)^{i-1}f'(X_j)\\
    \hline
    (i+1)f(X_1)^{i} f'(X_1)& \cdots & (i+1)f(X_i)^{i} f'(X_i) & \vline  & (i+1)f(X_j)^{i} f'(X_j)
\end{pmatrix}
\]
which satisfy that $\det(A_j(\boldsymbol{x}))=0$ for $j=i+1,\ldots,k$. Then, $\boldsymbol{x}$ vanishes in $Q_j= f'(x_j)\prod_{l=1}^i(f(x_j)-f(x_l))$ for $j=i+1,\ldots,k$. Taking $X_k>\cdots > X_1$ for a graded lexicographic order it turns out that $LT(Q_j(\boldsymbol{X}))=a_n^{i+1}n X_j^{n(i+1)-1}$ for $j=i+1,\ldots,k$.
Thus $LT(Q_j(\boldsymbol{X}))$, $i+1\leq j\leq k,$ are relatively prime and therefore, they form a
Gröbner basis of the ideal $J$ that they generate (see, e.g., \cite[\S 2.9, Proposition 4]{CLO92}).

Then
the initial of the ideal $J$ is generated by $LT(Q_{i+1}(\boldsymbol{X})),\ldots,LT(Q_k(\boldsymbol{X}))$, which form a regular sequence
of $\fq[X_1,\ldots,X_k]$. Therefore, by \cite[Proposition 15.15]{Eisenbud95}, the polynomials $Q_{i+1}(\boldsymbol{X}),\ldots,Q_k(\boldsymbol{X})$ form
a regular sequence of $\fq[X_1,\ldots,X_k]$. We conclude that $V(Q_{i+1}(\boldsymbol{X}),\ldots,Q_{k}(\boldsymbol{X}))$ is a set  theoretic complete intersection of $\A^k$ of dimension $i$. 
The rest of the proof is similar to the proof of Lemma \ref{lemma: codimension del lugar singula de V}. 
\end{proof}
In order to study the projective closure of $V_g$ we observe that for $1\leq j \leq m$ there exists $h_j\in \fq[X_0,\ldots,X_k]$ of degree $r_j$ such that 
$$g_j^{h}=a_n(X_1^{nj}+\cdots+ X_k^{nj})+X_0^{nj-r_j}h_j.$$
With the same arguments of Theorems \ref{teo: props de la clausura en el infinito}  and \ref{teo: pcl V interseccion completa}
we conclude that $\mathrm{pcl}(V_g) \subset \Pp^k$ is a complete intersection of dimension $k-m$, which singular locus has codimension at least $2$ and  $\mathrm{pcl}(V_g)\cap \{X_0=0\}\subset \Pp^{k-1}$ is a complete intersection of dimension $k-m-1$ and its singular locus has codimension  at least $2$. Then, we obtain the following estimate on the number of $\fq$-the rational solutions of the system \eqref{ sistema con los f}
$$\left ||\mathcal{S}|-q^{k-m}\right|\leq (3 m^2 n)^{k+1}\cdot q^{\frac{k}{2}}.$$
Now we have to estimate the number of $\fq$-the rational solutions of the system \eqref{ sistema con los f} for which there exists at least $i\neq j$ with $f(x_i)=f(x_j)$. Suppose that $i=1$ and $j=2$. In this case we have to consider the following system
\begin{equation}\label{ sistema con los f dos iguales}
\left\{ \begin{array}{lcc}
             2f(X_1)+f(X_3)+\cdots + f(X_k)=b_1 \\
            \hspace{50px}\vdots\\ 2f(X_1)^{m}+f(X_3)^{m}+\cdots + f(X_k)^{m}=b_m \\
             \end{array}.
   \right.    
\end{equation}
If $\mathcal{S}'$ is the set of $\fq$-rational solutions of \eqref{ sistema con los f dos iguales} with similar arguments as above, we can prove that
$$\left ||\mathcal{S}'|-q^{k-m-1}\right|\leq (3 m^2 n)^{k}\cdot q^{\frac{k-1}{2}}.$$
Now we can obtain a lower bound of $|B|$, in fact, we have that:
$$|B|\geq q^{k-m}-(3m^2 n)^{k+1}\cdot q^{\frac{k}{2}}-\binom{k}{2}(q^{k-m-1}+(3m^2 n)^{k}\cdot q^{\frac{k-1}{2}}).$$

From \eqref{ec: desigualdad 1 Brun} and \eqref{ec: Brun desigualdad 2} with $d_m=m\cdot n$ we obtain the following results.

\begin{theorem} \label{teo: existencia con Brun 2}
    Let $ m\cdot n <\frac{k-25}{50}$ and $k\leq q^{0,24}$
    then $N_m(k,\bfs{b},D)>0$ for all $\bfs{b}\in \fq^m$
    .
\end{theorem}
 Some families of polynomials for which we can apply Theorem \ref{teo: existencia con Brun 2}.
 \begin{itemize}
     \item If $q\geq \frac{149p}{p-1}$, $f=X^{p}-X^{p-1}$, $|D|=q-\frac{q}{p}$.
     \item If $q\geq 148, \,\,\gcd(n,q-1)+1$, $f=X^{n}$, $|D|= 1+ \frac{q-1}{\gcd(n,q-1)}$.
     \item If $q=2^s$, $s\geq 2$, $n$ an even number, let  $D_n(X,a)$ be the Dickson's polynomial of  degree $n$ and  parameter $a\in \fq$ and  $|D|=\frac{q-1}{2 \gcd(n,q-1)}+\frac{q+1}{2\gcd(n, q+1)}$.
     \item $p>2$, $q=p^s$, $n$ an odd number  with $p|n$ let $D_n(X,a)$ be the Dickson's polynomial of degree $n$ and parameter $a\in \fq$ and  $|D|=\frac{q-1}{2 \gcd(n,q-1)}+\frac{q+1}{2\gcd(n, q+1)}$.
 \end{itemize}

\subsection{Moment subset sums problem when $D=\{x^n: x\in \fq\}$}
In what follows we shall consider the Moment Subset Sums Problem when $D=\{x^n:\,\,x\in \fq\}$, $n$ and $q-1$ are relatively prime. In this case, we can provide an estimate on the number $N_m(k,\bfs{b}, D)$.
Observe that, since $q-1$ and $n$ are relatively prime,  the function $f:\fq\rightarrow \fq\/,\,f(x)=x^n$ is surjective. Therefore $N_m(k,\bfs{b}, D)>0$ if and only if there exists at least one $\fq$-rational solution  with distinct coordinates of the system:
\begin{equation}\label{eq subset sum problem}
\left\{ \begin{array}{lcc}
             a_1^{n}+a_2^{n}+\cdots + a_k^{n}=b_1 \\
            \hspace{50px}\vdots\\ a_1^{mn}+a_2^{mn}+\cdots + a_k^{mn}=b_m .\\
             \end{array}
   \right.    
\end{equation}

Then, we obtain we obtain analogous versions of Theorems \ref{teo: estimación Nm(k,b) p no divide a k} and \ref{teo: estimación Nm(k,b) p divide a k} taking  $d_m=m \cdot n$.

In \cite{Marino2020} the authors obtain an estimate on the number $N_m(k,\bfs{b},D)$ when $D=\{g(x):\,\, x\in \fq\}$ with $g(x)=x^n$ or $g$ is a Dickson's polynomial. More precisely, in the case $g(x)=x^n$, with $\gcd(n,q-1)=1$ ,they have that 
\begin{equation} \label{estimacion marino}
\left |N_m(k,\bfs{b}, D)-\frac{1}{q^m}{q\choose k}\right|\leq \binom{0.013 q +k+\frac{q}{p}}{k}.
   \end{equation} 
This estimate holds if $mn+1<0.013q^{\frac{1}{2}}$.

Observe that Theorems \ref{teo: estimación Nm(k,b) p no divide a k} and \ref{teo: estimación Nm(k,b) p divide a k} improves \eqref{estimacion marino} in several aspects. Indeed, on one hand, we determine one more term in the asymptotic development in terms of $q$, when $p$ divides $k$. On the other hand, if $q>>mn$ then, we obtain a better error term  than the one provided in \eqref{estimacion marino} in the sense that we shall explain. Namely, we prove that $N_m(k,\bfs{b}, D)=\frac{1}{q^m}{q\choose k}+\frac{v(\bfs{b})}{q^m} (-1)^{k+\frac{k}{p}} \binom{ q/p} { k/p  } + \mathcal{O}(p^{\frac{s}{2}k})$ instead of $N_m(k,\bfs{b}, D)= \frac{1}{q^m}{q\choose k}+\mathcal{O}(p^{sk})$.

Theorems  \ref{teo: existencia con estimación} and \ref{teo: existencia con estimación 2}
provide conditions over $m$, $n$, $k$ and $q$ which imply that $N_m(k,\bfs{b}, D)>0$ when the evaluation set is $D=\{x^n: x\in \fq\}$, $d_m=mn$ and $n$ and $q-1$ are relatively prime.
In \cite{Marino2020} the authors provide conditions which imply that $N_m(k,\bfs{b}, D)>0$. More precisely, when the evaluation set is $D=\{x^n: x\in \fq\}$ the authors obtain the following result.
\begin{theorem} \cite[Theorem 8, Corollary 2, Theorem 10]{Marino2020} \label{Teo: existencia marino cojunto evaluacion un monomio} $N_m(k,\bfs{b}, D)>0$ if  any of these conditions hold:
 \begin{itemize}
    \item If $k \leq 3m+1$ and $1 \leq m\leq q-1$.
    \item If  $2n(mn+1)<q^{1/6}$ and $3m+1< k < q^{5/12}$.
    \item If $p>2$, $mn+1 \leq 0.013 q^{1/2}$ and $6m \ln(q) \leq k \leq q/2$.
\end{itemize}
\end{theorem} 
We observe that Theorem \ref{Teo: existencia marino cojunto evaluacion un monomio} cannot guarantee that $N_m(k, \bfs{b}, D)>0$ for $k \in (q^{5/12}, 6m \ln(q))$ when $q$ satisfies $\frac{1}{q^{5/12}}\ln(q)>\frac{1}{6m}$. In contrast, Theorems \ref{teo: existencia con estimación} and  \ref{teo: existencia con estimación 2}
with $d_m=mn$  provide an answer for $m$ under the conditions of that result. 

On the other hand, if 
$q^{0.24} \in (3m+1, q^{5/12})$ and $k\in (3m+1,q^{0.24}]$ Theorem \ref{teo: existencia con Brun 2} let us to consider a wider range of $m$ than Theorem \ref {Teo: existencia marino cojunto evaluacion un monomio} such that $N_m(k,\bfs{b}, D)>0$.  Indeed, from Theorem \ref{teo: existencia con Brun 2}, we can say that $N_m(k,\bfs{b}, D)>0$  if  $m< \frac{k-25}{50n}<\frac{q^{0.24}}{50n}-\frac{25}{50n}$, while, from Theorem \ref{Teo: existencia marino cojunto evaluacion un monomio},  $N_m(k,\bfs{b}, D)>0$ holds if $m<\frac{q^{1/6}}{2n^2}-\frac{1}{n}$.



\section{Appendix A.}
The purpose of this appendix is to collect some results related to  combinatorial numbers which have been necessary in different proofs  of Section \ref{aplicaciones}. 
\begin{remark}\label{cota inferior de combinatorio}
If $k \leq q-\sqrt{q}+1$, then $$\Big(\frac{q}{\sqrt{q}+k-1}\Big)^k \leq \frac{\binom{q}{k}}{(-1)^k\binom{-\sqrt{q}}{k}} \leq q^{k/2}.$$  Indeed, it is well known that 
$$\binom{x}{y}=\frac{\Gamma(x+1)}{\Gamma(y+1)\Gamma(x-y+1)}, \quad \Gamma(x+1)=x\Gamma(x),$$ where $\Gamma$ is the \textit{Gamma function} and $x,y$ are complex numbers. 
Then we deduce that 
\begin{equation*}
    \binom{-\sqrt{q}}{k}=\frac{(-1)^k \sqrt{q}(\sqrt{q}+1)\cdots (\sqrt{q}+k-1)}{k!}.
\end{equation*}
Hence,
\begin{equation}
\frac{\binom{q}{k}}{(-1)^k\binom{-\sqrt{q}}{k}}=\frac{q(q-1) \cdots (q-(k-1))}{(\sqrt{q}+k-1)(\sqrt{q}+k-2) \cdots \sqrt{q} }.
\end{equation}
Then, the bounds of the  statement follows from the above equality and  from the fact that if $k \leq q-\sqrt{q}+1$,  $\frac{q-i}{\sqrt{q}+k-i-1}\geq \frac{q-j}{\sqrt{q}+k-j-1}$ holds for $0\leq j\leq i\leq k-1.$
\end{remark}


In the following remark, we shall provide a comment on the bounds of Theorems \ref{teo: estimación Nm(k,b) p no divide a k} and \ref{teo: estimación Nm(k,b) p  divide a k}.

\begin{remark} \label{Observacion 1}
Let $q \geq 2^{20}$, $m \leq \frac{k}{20}$ and $k\leq 2 q^{0.9}-\sqrt{q}+1$ then, from Remark  \ref{cota inferior de combinatorio}, we have that 
    \begin{equation}
        \frac{\binom{q}{k}}{(-1)^k\binom{-\sqrt{q}}{k}} \geq \frac{q^k}{(\sqrt{q}+k-1)^k} \geq \frac{q^{0.1k}}{2^k} \geq q^m
    \end{equation}
    holds.
    
\end{remark}

\begin{remark}\label{observacion para el caso p divide a k}
Suppose that $p$ divides $k$ and $k>p$. From simple calculations we derive that  
$$ \binom{q/p}{k/p} = \binom{q}{k}\cdot \frac{(k-1)}{(q-1)} \cdots \frac{(k-(p-1))}{(q-(p-1))} \cdot \frac{(k-(k-p+1))}{(q-(k-p+1))} \cdots \frac{(k-(k-1))}{(q-(k-1))}.$$
Since $\frac{k-i}{q-i}>\frac{k-j}{q-j}$ holds for $j>i$, we conclude that
$$\binom{q/p}{k/p} \leq \binom{q}{k} \Big(\frac{k-1}{q-1}\Big)^{2p-2} \leq \binom{q}{k}\Big(\frac{k}{q}\Big)^{2p-2} .$$
\end{remark}

\begin{remark} \label{remark 7.4} Suppose that $p$ divides $k$ and $p\geq 3$. If  $k \leq 2q^{0.9}-\sqrt{q}+1$, $q \geq  2^{21}$ and  $m \leq \frac{k}{20}$, then from Remarks \ref{Observacion 1} and \ref{observacion para el caso p divide a k} and by elementary calculation we can prove that 
\begin{equation*}
\frac{1}{q^m}\Big({q\choose k} +(-1)^{k+k/p} v(\bfs{b}) \binom{q/p}{k/p}) >(-1)^k \binom{-\sqrt{q}}{k}.
\end{equation*}
\end{remark}
Through elementary calculations, the following remark is obtained.
\begin{remark} 
\label{Observacion 2}
If $k<p$ and $m \leq \frac{k-\lfloor \frac{k}{p}\rfloor-1}{2}$ or $k>p$ and $m \leq \frac{k}{2}-\lfloor \frac{k}{p}\rfloor$, we have that $M_1 \cdot (-1)^k \binom{-\sqrt{q}}{k} k!  \geq d_m^m \, 2^{k+2m}\, q^{\frac{2m+1}{2}}$ and $M_1 \cdot (-1)^k \binom{-\sqrt{q}}{k}k!  \geq q^{m+\lfloor \frac{k}{p}\rfloor} \, 2^{k+2m+\lfloor \frac{k}{p}\rfloor+1},$ where $M_1=3^3 2^{m-1} (3+d_m m)^{k+1}.$
\end{remark}

\end{document}